\documentclass[10pt]{article}

\usepackage{amsmath,amsbsy,amssymb,amsthm}
\usepackage{a4wide}
\usepackage{graphicx,psfrag}
\usepackage{subfigure}

\newtheorem{Theorem}{Theorem}
\newtheorem{Definition}[Theorem]{Definition}
\newtheorem{Remark}[Theorem]{Remark}
\newtheorem{Lemma}[Theorem]{Lemma}

\def\a{{\alpha(t)}}
\def\ak{{\alpha_k(t_k)}}
\def\Dak{{\alpha'_k(t_k)}}
\def\t{\tau}
\def\LCI{{^C_aD_t^{\a}}}
\def\RCI{{^C_tD_b^{\a}}}
\def\LCII{{^C_a\mathcal{D}_t^{\a}}}
\def\RCII{{^C_t\mathcal{D}_b^{\a}}}
\def\LCIII{{^C_a\mathbb{D}_t^{\a}}}
\def\RCIII{{^C_t\mathbb{D}_b^{\a}}}
\def\PLCI{{^C_{a_k}D_{t_k}^{\ak}}}
\def\PRCI{{^C_{t_k}D_{b_k}^{\ak}}}
\def\PLCII{{^C_{a_k}\mathcal{D}_{t_k}^{\ak}}}
\def\PRCII{{^C_{t_k}\mathcal{D}_{b_k}^{\ak}}}
\def\PLCIII{{^C_{a_k}\mathbb{D}_{t_k}^{\ak}}}
\def\PRCIII{{^C_{t_k}\mathbb{D}_{b_k}^{\ak}}}
\def\C{\left(^{n-\ak}_{\quad p}\right)}
\def\D{\left(^{1-\ak}_{\quad p}\right)}
\def\DS{\displaystyle}

\begin{document}

\title{Caputo derivatives of fractional variable order:\\
numerical approximations\thanks{Part of first
author's Ph.D., which is carried out at the University of Aveiro
under the Doctoral Programme \emph{Mathematics and Applications}
of Universities of Aveiro and Minho.\newline
This is a preprint of a paper whose final and definite form is in
\emph{Communications in Nonlinear Science and Numerical Simulation},
ISSN: 1007-5704. Paper submitted 27/May/2015; revised 06/Oct/2015; 
accepted 30/Oct/2015.}}

\author{Dina Tavares$^{a,b}$\\
\texttt{dtavares@ipleiria.pt}
\and
Ricardo Almeida$^b$\\
\texttt{ricardo.almeida@ua.pt}
\and Delfim F. M. Torres$^b$\\
\texttt{delfim@ua.pt}}

\date{$^{a}${\em{ESECS, Polytechnic Institute of Leiria, 2410--272 Leiria, Portugal}}\\[0.3cm]
$^{b}${\em{\text{Center for Research and Development in Mathematics and Applications (CIDMA)},
Department of Mathematics, University of Aveiro, 3810--193 Aveiro, Portugal}}}

\maketitle


\begin{abstract}
We present a new numerical tool to solve partial differential equations
involving Caputo derivatives of fractional variable order.
Three Caputo-type fractional operators are considered,
and for each one of them an approximation formula is obtained in terms of standard
(integer-order) derivatives only. Estimations for the error of the approximations are also provided.
We then compare the numerical approximation of some test function with its exact fractional derivative.
We end with an exemplification of how the presented methods can be used
to solve partial fractional differential equations of variable order.

\bigskip


\medskip


\smallskip

\noindent \textbf{PACS 2010}: 02.30.Gp, 02.60.Lj.
\end{abstract}


\section{Introduction}

As is well known, several physical phenomena are often better described
by fractional derivatives \cite{MR3108996,MR2922940,MR3089091}.
This is mainly due to the fact that fractional operators
take into consideration the evolution of the system, by taking the global
correlation, and not only local characteristics.
Moreover, integer-order calculus sometimes
contradict the experimental results
and therefore derivatives of fractional
order may be more suitable \cite{MR1890104}.

An interesting recent generalization of the theory of fractional calculus consists
to allow the fractional order of the derivative
to be non-constant, depending on time \cite{MR3197190,Od0,Od1}.
With this approach, the non-local properties are more evident and numerous
applications have been found in physics, control and signal processing
\cite{Coimbra2,Ingman,Ostalczyk,Od2,Ramirez2,Rapaic,Valerio}.
One difficult issue, that usually arises when dealing with such fractional operators,
is the extreme difficulty in solving analytically such problems \cite{MR3047859,Zhuang}.
Thus, in most cases, we do not know the exact solution
for the problem and one needs to seek a numerical approximation.
Several numerical methods can be found in the literature,
typically applying some discretization over time or replacing the fractional
operators by a proper decomposition \cite{MR3047859,Zhuang}.

Recently, new approximation formulas were given for fractional constant order operators,
with the advantage that higher-order derivatives are not required
to obtain a good accuracy of the method \cite{Atanackovic1,Pooseh0,Pooseh1}.
These decompositions only depend on integer-order derivatives,
and by replacing the fractional operators that appear in the problem by them,
one leaves the fractional context ending up in the presence of a standard problem,
where numerous tools are available to solve them. Here we extend such decompositions
to Caputo fractional problems of variable order.

The paper is organized as follows. To start, in Section~\ref{sec:FC} we formulate
the needed definitions. Namely, we present three types of Caputo derivatives
of variable fractional order. First, we consider one independent variable only (Section~\ref{sec:2.1});
then we generalize for several independent variables (Section~\ref{sec:2.2}). Section~\ref{sec:theorems}
is the main core of the paper: we prove approximation formulas for the given fractional operators
and upper bound formulas for the errors. To test the efficiency of the proposed method,
in Section~\ref{sec:EX} we compare the exact fractional derivative of some test function
with the numerical approximations obtained from the decomposition formulas given in Section~\ref{sec:theorems}.
To end, in Section~\ref{sec:app} we apply our method to approximate two physical problems
involving Caputo fractional operators of variable order (a time-fractional diffusion equation
in Section~\ref{example1} and a fractional Burgers' partial differential equation in fluid mechanics
in Section~\ref{sec:fluid:mech}) by classical problems that may be solved by well-known standard techniques.


\section{Fractional calculus of variable order}
\label{sec:FC}

In the literature of fractional calculus, several different
definitions  of derivatives are found \cite{MR1347689}.
One of those, introduced by Caputo in 1967 \cite{Caputo} and studied independently by other authors,
like D\v zrba\v sjan and Nersesjan in 1968 \cite{Dzherbashyan} and Rabotnov in 1969 \cite{Rabotnov},
has found many applications and seems to be more suitable
to model physic phenomena \cite{Coimbra,Dalir,Diethelm,Machado,Murio,Singh,Sweilam,Yajima}.
Before generalizing the Caputo derivative for a variable order of differentiation,
we recall two types of special functions: the Gamma and Psi functions.
The Gamma function is an extension of the factorial function to real numbers, and is defined by
$$
\Gamma(t)=\int_0^\infty \t^{t-1}\exp(-\t)\,d\t, \quad t>0.
$$
We mention that other definitions exist for the Gamma function, and it is possible
to define it for complex numbers, except the non-positive integers.
A basic but fundamental property that we will use later is the following:
$$
\Gamma(t+1)=t\, \Gamma(t).
$$
The Psi function is the derivative of the logarithm of the Gamma function:
$$
\Psi(t)=\frac{d}{dt}\ln\left(\Gamma(t)\right)=\frac{\Gamma'(t)}{\Gamma(t)}.
$$
Given $\alpha\in(0,1)$, the left and right Caputo fractional derivatives
of order $\alpha$ of a function $x:[a,b]\to\mathbb{R}$ are defined by
$$
{_a^CD_t^\alpha}x(t)={_aD_t^\alpha}(x(t)-x(a))
$$
and
$$
{_t^CD_b^\alpha}x(t)={_tD_b^\alpha}(x(t)-x(b)),
$$
respectively, where ${_aD_t^\alpha}x(t)$ and ${_tD_b^\alpha}x(t)$ denote
the left and right Riemann--Liouville fractional derivative of order $\alpha$, that is,
$$
{_aD_t^\alpha}x(t)=\frac{1}{\Gamma(1-\alpha)}\frac{d}{dt}\int_a^t(t-\t)^{-\alpha}x(\t)d\t
$$
and
$$
{_tD_b^\alpha}x(t)=\frac{-1}{\Gamma(1-\alpha)}\frac{d}{dt}\int_t^b(\t-t)^{-\alpha}x(\t)d\t.
$$
If $x$ is differentiable, then, integrating by parts,
one can prove the following equivalent definitions:
$$
{_a^CD_t^\alpha}x(t)=\frac{1}{\Gamma(1-\alpha)}\int_a^t(t-\t)^{-\alpha}x'(\t)d\t
$$
and
$$
{_t^CD_b^\alpha}x(t)=\frac{-1}{\Gamma(1-\alpha)}\int_t^b(\t-t)^{-\alpha}x'(\t)d\t.
$$
From these definitions, it is clear that the Caputo fractional derivative of a constant is zero,
which is false when we consider the Riemann--Liouville fractional derivative. Also,
the boundary conditions that appear in the Laplace transform of the Caputo derivative
depend on integer-order derivatives, and so coincide with the classical case.


\subsection{Variable order Caputo derivatives for functions of one variable}
\label{sec:2.1}

Our goal is to consider fractional derivatives of variable order, with $\alpha$ depending on time.
In fact, some phenomena in physics are better described
when the order of the fractional operator is not constant,
for example, in the diffusion process in an inhomogeneous or heterogeneous medium, or processes
where the changes in the environment modify the dynamic of the particle \cite{Chechkin,Santamaria,Sun}.
Motivated by the above considerations, we introduce three types of Caputo fractional derivatives.
The order of the derivative is considered as a function $\a$ taking values on the open interval $(0,1)$.
To start, we define two different kinds of Riemann--Liouville fractional derivatives.

\begin{Definition}[Riemann--Liouville fractional derivatives of order $\a$---types I and II]
Given a function  $x:[a,b]\to\mathbb{R}$,
\begin{enumerate}
\item the type I left Riemann--Liouville fractional derivative of order $\a$ is defined by
$$
{_aD_t^\a}x(t)=\frac{1}{\Gamma(1-\a)}\frac{d}{dt}\int_a^t(t-\t)^{-\a}x(\t)d\t;
$$
\item the type I right Riemann--Liouville fractional derivative of order $\a$ is defined by
$$
{_tD_b^\a}x(t)=\frac{-1}{\Gamma(1-\a)}\frac{d}{dt}\int_t^b(\t-t)^{-\a}x(\t)d\t;
$$
\item the type II left Riemann--Liouville fractional derivative of order $\a$ is defined by
$$
{_a\mathcal{D}_t^\a}x(t)=\frac{d}{dt}\left(\frac{1}{\Gamma(1-\a)}\int_a^t(t-\t)^{-\a}x(\t)d\t\right);
$$
\item the type II right Riemann--Liouville fractional derivative of order $\a$ is defined by
$$
{_t\mathcal{D}_b^\a}x(t)=\frac{d}{dt}\left(\frac{-1}{\Gamma(1-\a)}\int_t^b(\t-t)^{-\a}x(\t)d\t\right).
$$
\end{enumerate}
\end{Definition}

The Caputo derivatives are given using the previous Riemann--Liouville fractional derivatives.

\begin{Definition}[Caputo fractional derivatives of order $\a$---types I, II and III]
Given a function  $x:[a,b]\to\mathbb{R}$,
\begin{enumerate}
\item the type I left Caputo derivative of order $\a$ is defined by
$$
\LCI x(t)={_aD_t^\a}(x(t)-x(a))
=\frac{1}{\Gamma(1-\a)}\frac{d}{dt}\int_a^t(t-\t)^{-\a}[x(\t)-x(a)]d\t;
$$
\item the type I right Caputo derivative of order $\a$ is defined by
$$
\RCI x(t)={_tD_b^\a}(x(t)-x(b))
=\frac{-1}{\Gamma(1-\a)}\frac{d}{dt}\int_t^b(\t-t)^{-\a}[x(\t)-x(b)]d\t;
$$
\item the type II left Caputo derivative of order $\a$ is defined by
$$
\LCII x(t)= {_a\mathcal{D}_t^\a} (x(t)-x(a))
=\frac{d}{dt}\left(\frac{1}{\Gamma(1-\a)}\int_a^t(t-\t)^{-\a}[x(\t)-x(a)]d\t\right);
$$
\item the type II right Caputo derivative of order $\a$ is defined by
$$
\RCII x(t)= {_t\mathcal{D}_b^\a}(x(t)-x(b))
=\frac{d}{dt}\left(\frac{-1}{\Gamma(1-\a)}\int_t^b(\t-t)^{-\a}[x(\t)-x(b)]d\t\right);
$$
\item the type III left Caputo derivative of order $\a$ is defined by
$$
\LCIII x(t)=\frac{1}{\Gamma(1-\a)}\int_a^t(t-\t)^{-\a}x'(\t)d\t;
$$
\item the type III right Caputo derivative of order $\a$ is defined by
$$
\RCIII x(t)=\frac{-1}{\Gamma(1-\a)}\int_t^b(\t-t)^{-\a}x'(\t)d\t.
$$
\end{enumerate}
\end{Definition}

In contrast with the case when $\alpha$ is a constant,
definitions of different types do not coincide.

\begin{Theorem}
The following relations between the left fractional operators hold:
\begin{equation}
\label{eq1}
\LCI x(t)=\LCIII x(t)+\frac{\alpha'(t)}{\Gamma(2-\a)}
\int_a^t(t-\t)^{1-\a}x'(\t)\left[\frac{1}{1-\a}-\ln(t-\t)\right]d\t
\end{equation}
and
\begin{equation}
\label{eq2}
\LCI x(t)=\LCII x(t)-\frac{\alpha'(t)\Psi(1-\a)}{\Gamma(1-\a)}
\int_a^t(t-\t)^{-\a}[x(\t)-x(a)]d\t.
\end{equation}
\end{Theorem}

\begin{proof}
Integrating by parts, one gets
\begin{equation*}
\begin{split}
\LCI x(t)&= \DS\frac{1}{\Gamma(1-\a)}
\frac{d}{dt}\int_a^t(t-\t)^{-\a}[x(\t)-x(a)]d\t\\
&= \DS\frac{1}{\Gamma(1-\a)}\frac{d}{dt}
\left[\frac{1}{1-\a}\int_a^t(t-\t)^{1-\a}x'(\t)d\t\right].
\end{split}
\end{equation*}
Differentiating the integral, it follows that
\begin{equation*}
\begin{split}
\LCI x(t)& \DS=\frac{1}{\Gamma(1-\a)}\left[
\frac{\alpha'(t)}{(1-\a)^2}\int_a^t(t-\t)^{1-\a}x'(\t)d\t\right.\\
&\DS\quad\left.+\frac{1}{1-\a}\int_a^t(t-\t)^{1-\a}x'(\t)\left[
-\alpha'(t)\ln(t-\t)+\frac{1-\a}{t-\t}\right]d\t\right]\\
&=\DS\LCIII x(t)+\frac{\alpha'(t)}{\Gamma(2-\a)}\int_a^t
(t-\t)^{1-\a}x'(\t)\left[\frac{1}{1-\a}-\ln(t-\t)\right]d\t.
\end{split}
\end{equation*}
The second formula follows from direct calculations.
\end{proof}

Therefore, when the order $\a \equiv c$ is a constant,
or for constant functions $x(t) \equiv k$, we have
$$
\LCI x(t)=\LCII x(t)=\LCIII x(t).
$$
Similarly, we obtain the next result.

\begin{Theorem}
The following relations between the right fractional operators hold:
$$
\RCI x(t) \DS=\DS\RCIII x(t)+\frac{\alpha'(t)}{\Gamma(2-\a)}
\int_t^b(\t-t)^{1-\a}x'(\t)\left[\frac{1}{1-\a}-\ln(\t-t)\right]d\t
$$
and
$$
\RCI x(t)=\RCII x(t)+\frac{\alpha'(t)\Psi(1-\a)}{\Gamma(1-\a)}
\int_t^b(\t-t)^{-\a}[x(\t)-x(b)]d\t.
$$
\end{Theorem}

\begin{Theorem}
\label{initialpoint}
Let $x\in C^1\left([a,b],\mathbb{R}\right)$. At $t=a$
$$
\LCI x(t)=\LCII x(t)=\LCIII x(t)=0;
$$
at $t=b$
$$
\RCI x(t)=\RCII x(t)=\RCIII x(t)=0.
$$
\end{Theorem}

\begin{proof}
We start proving the third equality at the initial time $t=a$. We simply note that
$$
\left|\LCIII x(t)\right|\leq \frac{\|x'\|}{\Gamma(1-\a)}\int_a^t(t-\t)^{-\a}d\t
=\frac{\|x'\|}{\Gamma(2-\a)}(t-a)^{1-\a},
$$
which is zero at $t=a$. For the first equality at $t=a$,
using equation \eqref{eq1}, and the two next relations
$$
\left|\int_a^t(t-\t)^{1-\a}\frac{x'(\t)}{1-\a}d\t\right|
\leq \frac{\|x'\|}{(1-\a)(2-\a)}(t-a)^{2-\a}
$$
and
$$
\left|\int_a^t(t-\t)^{1-\a}x'(\t)\ln(t-\t)d\t\right|
\leq \frac{\|x'\|}{2-\a}(t-a)^{2-\a}\left|\ln(t-a)-\frac{1}{2-\a}\right|,
$$
this latter inequality obtained from integration by parts,
we prove that $\LCI x(t)=0$ at $t=a$. Finally, we prove the second equality at $t=a$
by considering equation \eqref{eq2}: performing an integration by parts, we get
$$
\left|\int_a^t(t-\t)^{-\a}[x(\t)-x(a)]d\t\right|\leq \frac{\|x'\|}{(1-\a)(2-\a)}(t-a)^{2-\a}
$$
and so $\LCII x(t)=0$ at $t=a$. The proof that the right
fractional operators also vanish at the end point $t=b$
follows by similar arguments.
\end{proof}

With some computations, a relationship between the Riemann--Liouville
and the Caputo fractional derivatives is easily deduced:
\begin{equation*}
\begin{split}
{_aD_t^\a}x(t)&=\displaystyle\LCI x(t)
+\frac{x(a)}{\Gamma(1-\a)}\frac{d}{dt}\int_a^t(t-\t)^{-\a}d\t\\
&=\displaystyle\LCI x(t)+\frac{x(a)}{\Gamma(1-\a)}(t-a)^{-\a}\\
&\qquad\displaystyle+\frac{x(a)\alpha'(t)}{\Gamma(2-\a)}
(t-a)^{1-\a}\left[\frac{1}{1-\a}-\ln(t-a)\right]
\end{split}
\end{equation*}
and
\begin{equation*}
\begin{split}
{_a\mathcal{D}_t^\a}x(t)&=\displaystyle\LCII x(t)
+x(a)\frac{d}{dt}\left(\frac{1}{\Gamma(1-\a)}\int_a^t(t-\t)^{-\a}d\t\right)\\
&=\displaystyle\LCII x(t)+\frac{x(a)}{\Gamma(1-\a)}(t-a)^{-\a}\\
&\qquad\displaystyle+\frac{x(a)\alpha'(t)}{\Gamma(2-\a)}
(t-a)^{1-\a}\left[\Psi(2-\a)-\ln(t-a)\right].
\end{split}
\end{equation*}
For the right fractional operators, we have
\begin{equation*}
\begin{split}
{_tD_b^\a}x(t)&=\displaystyle\RCI x(t)+\frac{x(b)}{\Gamma(1-\a)}(b-t)^{-\a}\\
&\qquad\displaystyle-\frac{x(b)\alpha'(t)}{\Gamma(2-\a)}
(b-t)^{1-\a}\left[\frac{1}{1-\a}-\ln(b-t)\right]
\end{split}
\end{equation*}
and
\begin{equation*}
\begin{split}
{_t\mathcal{D}_b^\a}x(t)&=\displaystyle\RCII x(t)+\frac{x(b)}{\Gamma(1-\a)}(b-t)^{-\a}\\
&\qquad\displaystyle-\frac{x(b)\alpha'(t)}{\Gamma(2-\a)}
(b-t)^{1-\a}\left[\Psi(2-\a)-\ln(b-t)\right].
\end{split}
\end{equation*}
Thus, it is immediate to conclude that if $x(a)=0$, then
$$
{_aD_t^\a}x(t)=\LCI x(t) \quad \mbox{and} \quad {_a\mathcal{D}_t^\a}x(t)=\LCII x(t)
$$
and if $x(b)=0$, then
$$
{_tD_b^\a}x(t)=\RCI x(t) \quad \mbox{and} \quad {_t\mathcal{D}_b^\a}x(t)=\RCII x(t).
$$

Next we obtain formulas for the Caputo fractional derivatives of a power function.

\begin{Lemma}
\label{LemmaEx}
Let $x(t)=(t-a)^\gamma$ with $\gamma>0$. Then,
\begin{equation*}
\begin{split}
\LCI x(t) &=\DS \frac{\Gamma(\gamma+1)}{\Gamma(\gamma-\a+1)}(t-a)^{\gamma-\a}\\
&\DS\quad -\alpha'(t)\frac{\Gamma(\gamma+1)}{\Gamma(\gamma-\a+2)}(t-a)^{\gamma-\a+1}\\
&\DS \quad \times \left[\ln(t-a)-\Psi(\gamma-\a+2)+\Psi(1-\a)\right],\\
\LCII x(t) &=\DS \frac{\Gamma(\gamma+1)}{\Gamma(\gamma-\a+1)}(t-a)^{\gamma-\a}\\
&\DS\quad -\alpha'(t)\frac{\Gamma(\gamma+1)}{\Gamma(\gamma-\a+2)}(t-a)^{\gamma-\a+1}\\
&\DS \quad \times \left[\ln(t-a)-\Psi(\gamma-\a+2)\right],\\
\LCIII x(t) &=\DS  \frac{\Gamma(\gamma+1)}{\Gamma(\gamma-\a+1)}(t-a)^{\gamma-\a}.
\end{split}
\end{equation*}
\end{Lemma}

\begin{proof}
The formula for $\LCI x(t)$ follows immediately
from \cite{SamkoRoss}. For the second equality, one has
\begin{equation*}
\begin{split}
\LCII x(t) &=\DS\frac{d}{dt}\left(\frac{1}{\Gamma(1-\a)}
\int_a^t(t-\t)^{-\a}(\t-a)^{\gamma} d\t\right)\\
 &=\DS\frac{d}{dt}\left(\frac{1}{\Gamma(1-\a)}\int_a^t
 (t-a)^{-\a}\left(1-\frac{\t-a}{t-a}\right)^{-\a}(\t-a)^{\gamma} d\t\right).
\end{split}
\end{equation*}
With the change of variables $\t-a=s(t-a)$,
and with the help of the Beta function $B(\cdot,\cdot)$, we prove that
\begin{equation*}
\begin{split}
\LCII x(t) &=\DS\frac{d}{dt}\left(\frac{(t-a)^{-\a}}{\Gamma(1-\a)}
\int_0^1(1-s)^{-\a}s^{\gamma}(t-a)^{\gamma+1} ds\right)\\
&=\DS\frac{d}{dt}\left(\frac{(t-a)^{\gamma-\a+1}}{\Gamma(1-\a)}B(\gamma+1,1-\a)\right)\\
 &= \DS\frac{d}{dt}\left(\frac{\Gamma(\gamma+1)}{\Gamma(\gamma-\a+2)}(t-a)^{\gamma-\a+1}\right).
\end{split}
\end{equation*}
We obtain the desired formula by differentiating this latter expression.
The last equality follows in a similar way.
\end{proof}

Analogous relations to those of Lemma~\ref{LemmaEx},
for the right Caputo fractional derivatives of variable order,
are easily obtained.

\begin{Lemma}
\label{LemmaEx2}
Let $x(t)=(b-t)^\gamma$ with $\gamma>0$. Then,
\begin{equation*}
\begin{split}
\RCI x(t) &=\DS \frac{\Gamma(\gamma+1)}{\Gamma(\gamma-\a+1)}(b-t)^{\gamma-\a}\\
&\DS\quad +\alpha'(t)\frac{\Gamma(\gamma+1)}{\Gamma(\gamma-\a+2)}(b-t)^{\gamma-\a+1}\\
&\DS \quad \times \left[\ln(b-t)-\Psi(\gamma-\a+2)+\Psi(1-\a)\right],\\
\RCII x(t) &=\DS \frac{\Gamma(\gamma+1)}{\Gamma(\gamma-\a+1)}(b-t)^{\gamma-\a}\\
&\DS\quad +\alpha'(t)\frac{\Gamma(\gamma+1)}{\Gamma(\gamma-\a+2)}(b-t)^{\gamma-\a+1}\\
&\DS \quad \times \left[\ln(b-t)-\Psi(\gamma-\a+2)\right],\\
\RCIII x(t) &=\DS  \frac{\Gamma(\gamma+1)}{\Gamma(\gamma-\a+1)}(b-t)^{\gamma-\a}.
\end{split}
\end{equation*}
\end{Lemma}

With Lemma~\ref{LemmaEx} in mind, we immediately see that
$\LCI x(t)\not=\LCII x(t)\not=\LCIII x(t)$. Also,
at least for the power function, it suggests that $\LCIII x(t)$
may be a more suitable inverse operation
of the fractional integral when the order is variable.
For example, consider functions $x(t)=t^2$ and $y(t)=(1-t)^2$,
and the fractional order $\a=\frac{5t+1}{10}$, $t\in[0,1]$.
Then, $0.1\leq \a \leq 0.6$ for all $t$. Next we compare
the fractional derivatives of $x$ and $y$ of order $\a$
with the fractional derivatives of constant order $\alpha=0.1$ and $\alpha=0.6$.
By Lemma~\ref{LemmaEx}, we know that the left Caputo fractional
derivatives of order $\a$ of $x$ are given by
\begin{equation*}
\begin{split}
{^C_0D_t^{\a}} x(t)
&=\DS \frac{2}{\Gamma(3-\a)}t^{2-\a}
-\frac{t^{3-\a}}{\Gamma(4-\a)}\left[\ln(t)-\Psi(4-\a)+\Psi(1-\a)\right],\\
{^C_0\mathcal{D}_t^{\a}} x(t) &=\DS \frac{2}{\Gamma(3-\a)}t^{2-\a}
-\frac{t^{3-\a}}{\Gamma(4-\a)}\left[\ln(t)-\Psi(4-\a)\right],\\
{^C_0\mathbb{D}_t^{\a}} x(t) &=\DS  \frac{2}{\Gamma(3-\a)}t^{2-\a},
\end{split}
\end{equation*}
while by Lemma~\ref{LemmaEx2}, the right Caputo
fractional derivatives of order $\a$ of $y$  are given by
\begin{equation*}
\begin{split}
{^C_tD_1^{\a}} y(t) &=\DS \frac{2}{\Gamma(3-\a)}(1-t)^{2-\a}
+\frac{(1-t)^{3-\a}}{\Gamma(4-\a)}\left[\ln(1-t)-\Psi(4-\a)+\Psi(1-\a)\right],\\
{^C_t\mathcal{D}_1^{\a}} y(t) &=\DS \frac{2}{\Gamma(3-\a)}(1-t)^{2-\a}
+\frac{(1-t)^{3-\a}}{\Gamma(4-\a)}\left[\ln(1-t)-\Psi(4-\a)\right],\\
{^C_t\mathbb{D}_1^{\a}} y(t) &=\DS  \frac{2}{\Gamma(3-\a)}(1-t)^{2-\a}.
\end{split}
\end{equation*}
For a constant order $\alpha$, we have
$$
{^C_0D_t^{\alpha}} x(t) = \frac{2}{\Gamma(3-\alpha)}t^{2-\alpha}
\quad \mbox{and} \quad {^C_tD_1^{\alpha}} y(t)
=\frac{2}{\Gamma(3-\alpha)}(1-t)^{2-\alpha}.
$$
The results can be seen in Figure~\ref{comparison}.
\begin{figure}[!ht]
\begin{center}
\subfigure[${^C_0{D}_t^{\alpha(t)}} x(t)$]{\includegraphics[scale=0.3]{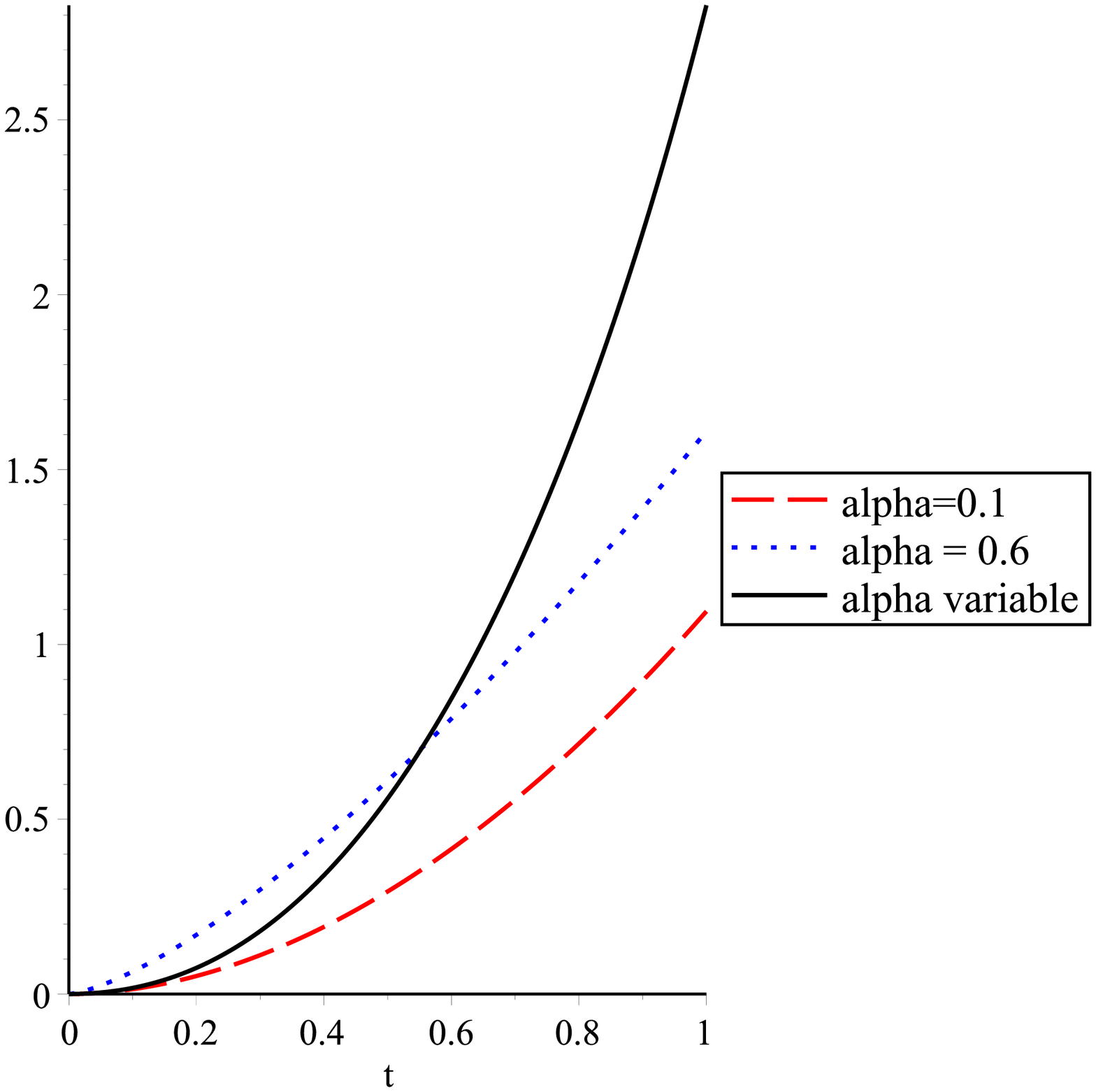}} \hspace{1cm}
\subfigure[${^C_0\mathcal{D}_t^{\alpha(t)}} x(t)$]{\includegraphics[scale=0.3]{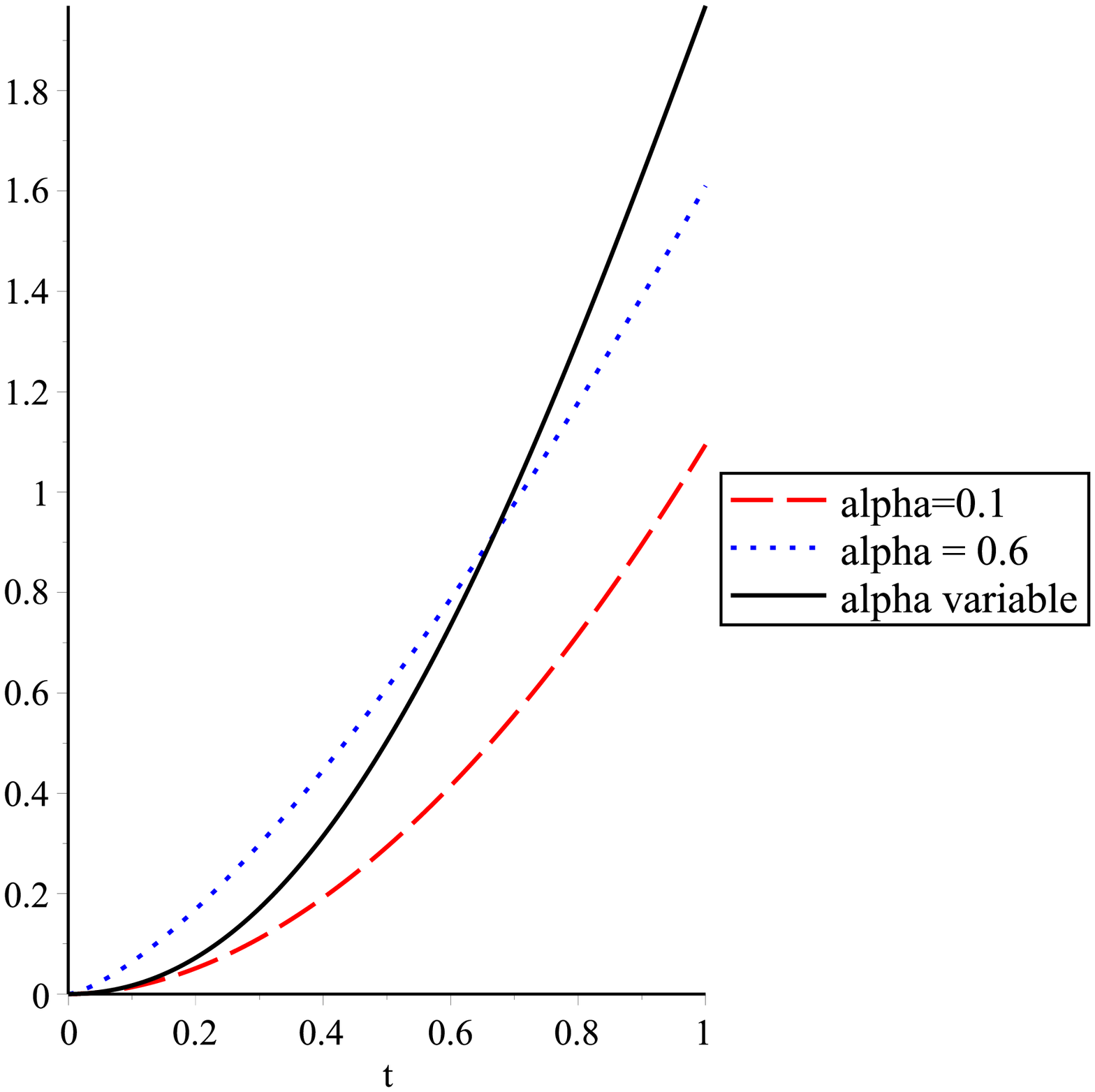}}
\subfigure[${^C_0\mathbb{D}_t^{\alpha(t)}} x(t)$]{\includegraphics[scale=0.3]{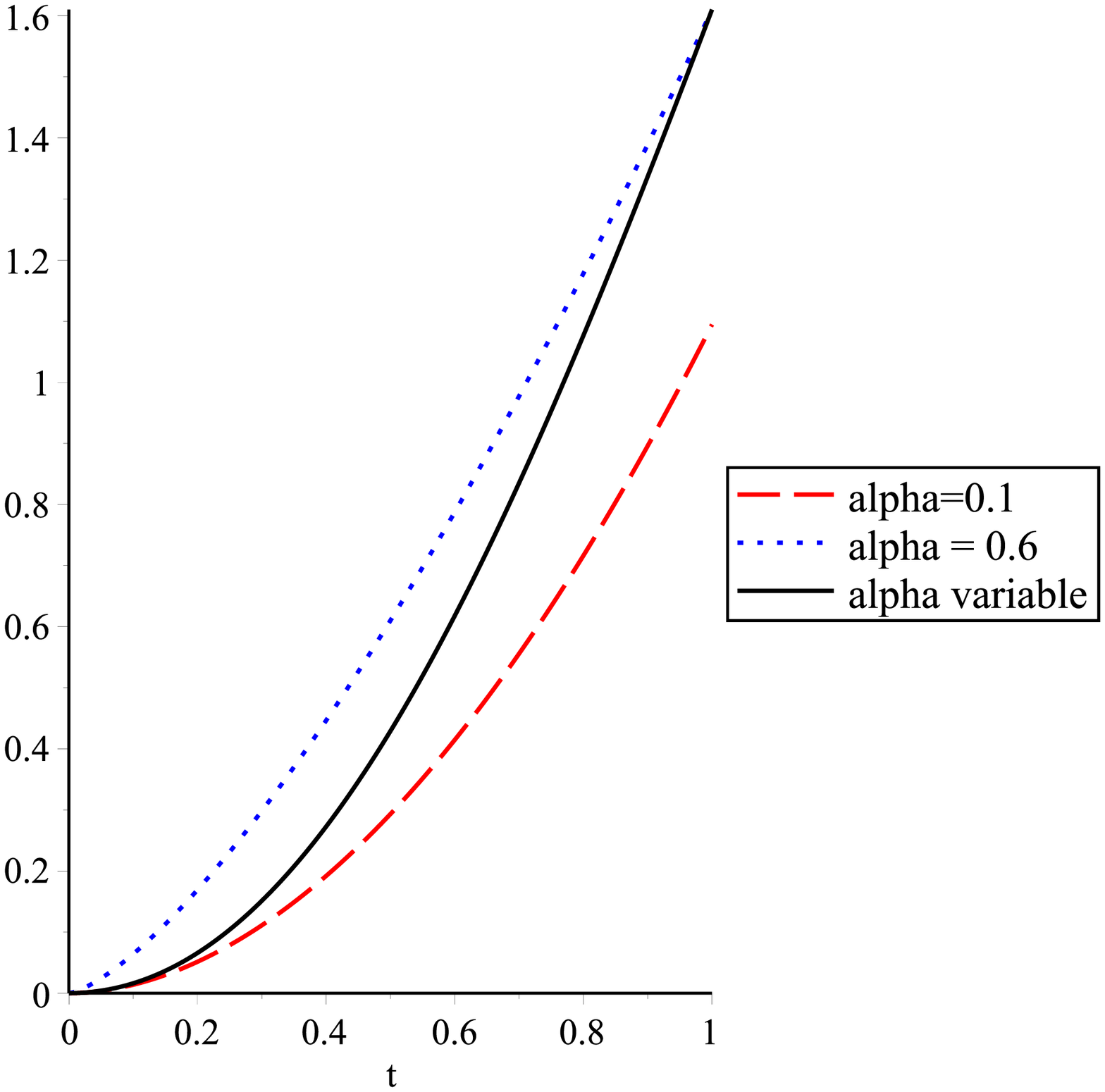}} \hspace{1cm}
\subfigure[${^C_t{D}_1^{\alpha(t)}} y(t)$]{\includegraphics[scale=0.3]{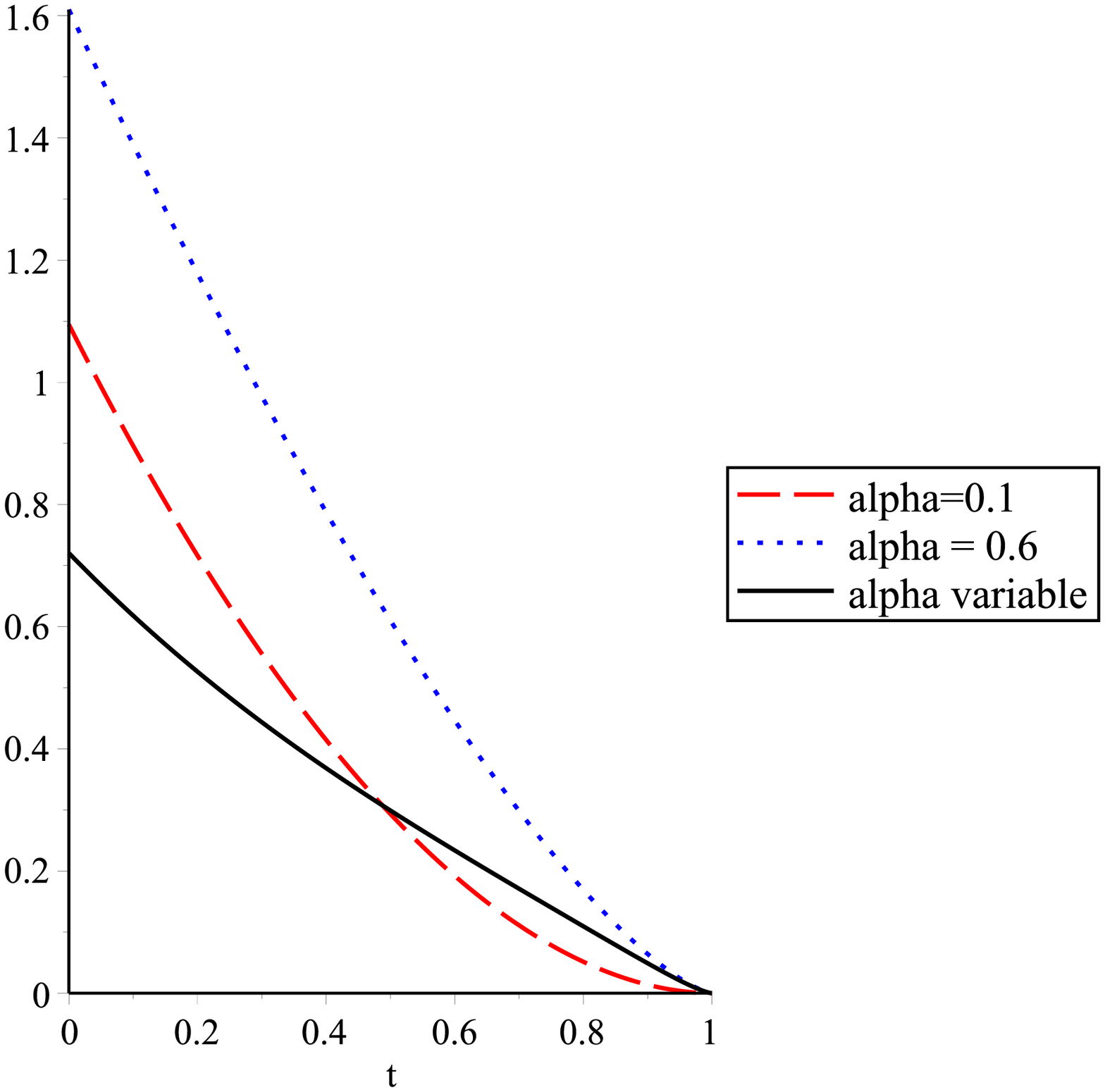}}
\subfigure[${^C_t\mathcal{D}_1^{\alpha(t)}} y(t)$]{\includegraphics[scale=0.3]{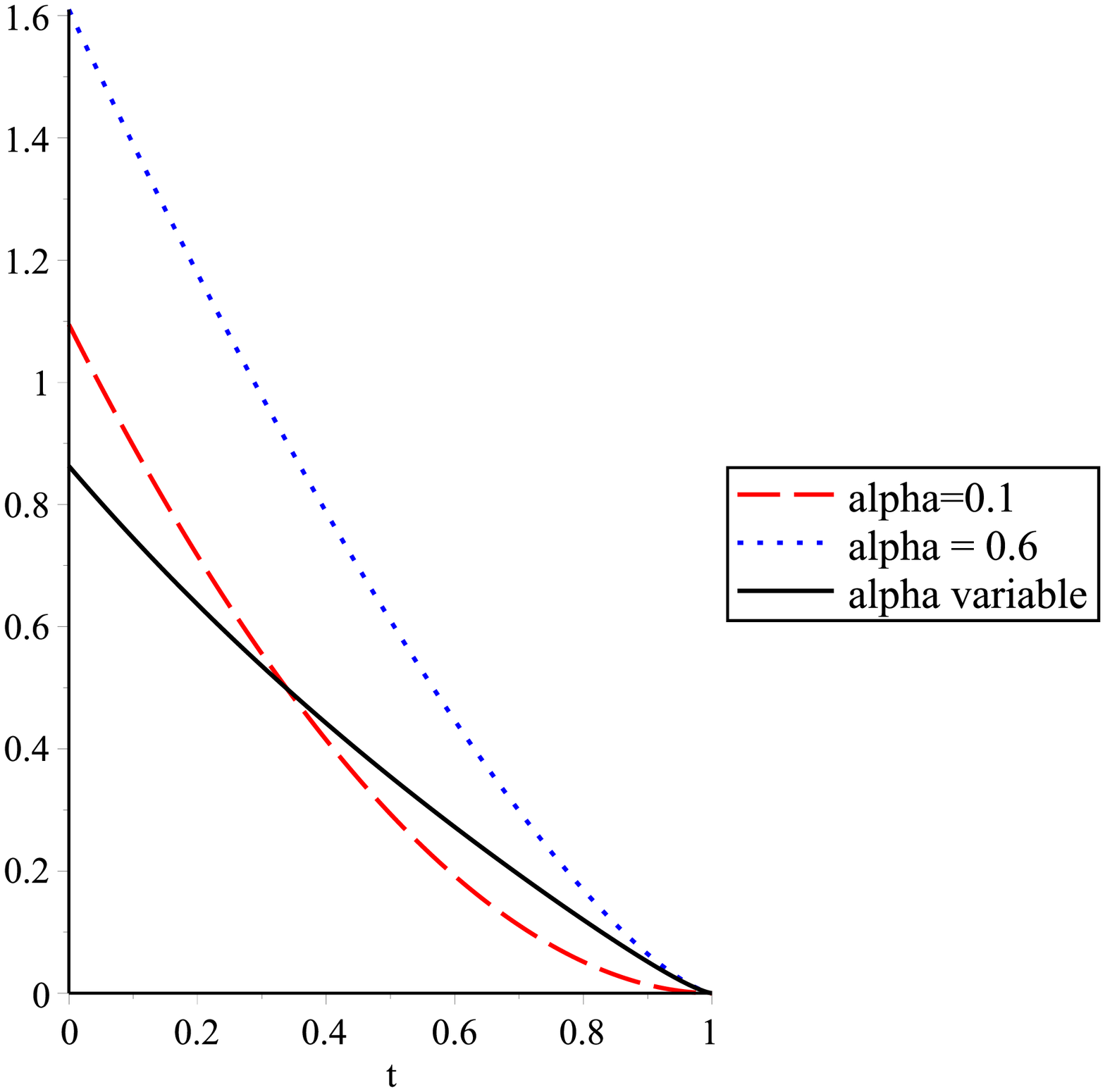}} \hspace{1cm}
\subfigure[${^C_t\mathbb{D}_1^{\alpha(t)}} y(t)$]{\includegraphics[scale=0.3]{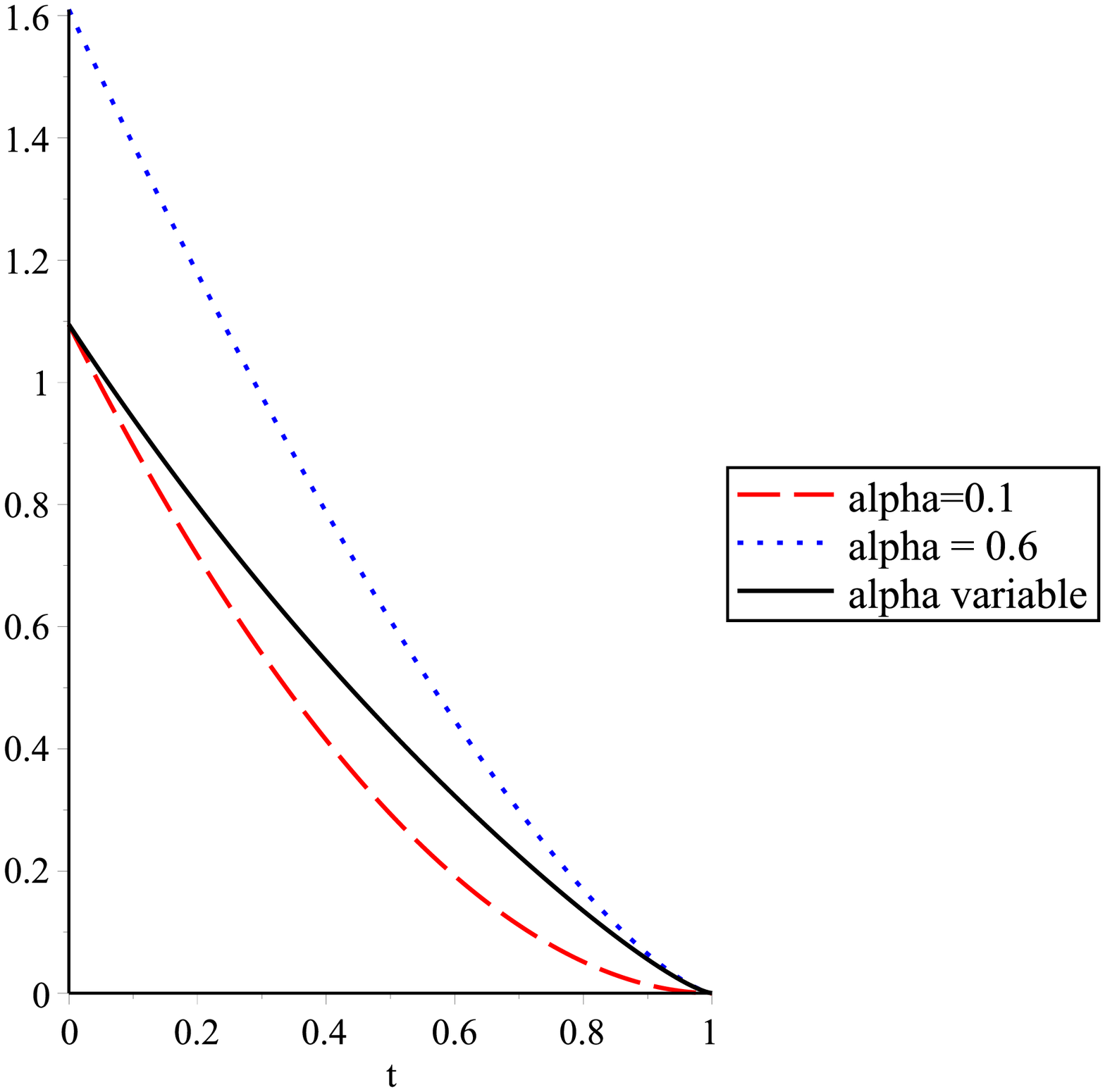}}
\end{center}
\caption{Comparison between variable order and constant order fractional derivatives.}\label{comparison}
\label{IntExp1}
\end{figure}


\subsection{Variable order Caputo derivatives for functions of several variables}
\label{sec:2.2}

Partial fractional derivatives are a natural extension
and are defined in a similar way. Let $m\in\mathbb{N}$,
$k\in\{1,\ldots,m\}$, and consider a function
$\DS x:\prod_{i=1}^m[a_i,b_i]\to\mathbb{R}$
with $m$ variables. For simplicity, we define the vectors
$$
[\t]_k=(t_1,\ldots,t_{k-1},\t,t_{k+1},\ldots,t_m)\in\mathbb{R}^m
$$
and
$$
(\overline t)=(t_1,\ldots,t_m)\in\mathbb{R}^m.
$$

\begin{Definition}[Partial Caputo fractional derivatives of variable order---types I, II and III]
\label{def:8}
Given a function $x:\prod_{i=1}^m[a_i,b_i]\to\mathbb{R}$
and fractional orders $\alpha_k:[a_k,b_k]\to(0,1)$,
$k\in\{1,\ldots,m\}$,
\begin{enumerate}
\item the type I partial left Caputo derivative of order $\ak$ is defined by
$$
\PLCI x(\overline t)=\frac{1}{\Gamma(1-\ak)}\frac{\partial}{\partial t_k}
\int_{a_k}^{t_k}(t_k-\t)^{-\ak}\left(x[\t]_k-x[a_k]_k\right)d\t;
$$
\item the type I partial right Caputo derivative of order $\ak$ is defined by
$$
\PRCI x(\overline t)=\frac{-1}{\Gamma(1-\ak)}\frac{\partial }{\partial t_k}
\int_{t_k}^{b_k}(\t-t_k)^{-\ak}\left(x[\t]_k-x[b_k]_k\right)d\t;
$$
\item the type II partial left Caputo derivative of order $\ak$ is defined by
$$
\PLCII x(\overline t)=\frac{\partial }{\partial t_k}\left(\frac{1}{\Gamma(1-\ak)}
\int_{a_k}^{t_k}(t_k-\t)^{-\ak}\left(x[\t]_k-x[a_k]_k\right)d\t\right);
$$
\item the type II partial right Caputo derivative of order $\ak$ is defined by
$$
\PRCII x(\overline t)=\frac{\partial }{\partial t_k}\left(\frac{-1}{\Gamma(1-\ak)}
\int_{t_k}^{b_k}(\t-t_k)^{-\ak}\left(x[\t]_k-x[b_k]_k\right)d\t\right);
$$
\item the type III partial left Caputo derivative of order $\ak$ is defined by
$$
\PLCIII x(\overline t)=\frac{1}{\Gamma(1-\ak)}
\int_{a_k}^{t_k}(t_k-\t)^{-\ak}\frac{\partial x}{\partial t_k}[\t]_kd\t;
$$
\item the type III partial right Caputo derivative of order $\ak$ is defined by
$$
\PRCIII x(\overline t)=\frac{-1}{\Gamma(1-\ak)}
\int_{t_k}^{b_k}(\t-t_k)^{-\ak}\frac{\partial x}{\partial t_k}[\t]_kd\t.
$$
\end{enumerate}
\end{Definition}

Similarly as done before, relations between these definitions can be proven.

\begin{Theorem}
The following four formulas hold:
\begin{multline}
\label{relation}
\PLCI x(\overline t)=\PLCIII x(\overline t)\\
+\frac{\Dak}{\Gamma(2-\ak)}\int_{a_k}^{t_k}(t_k-\t)^{1-\ak}
\frac{\partial x}{\partial t_k}[\t]_k\left[\frac{1}{1-\ak}-\ln(t_k-\t)\right]d\t,
\end{multline}
\begin{equation}
\label{relation3}
\PLCI x(\overline t)=\PLCII x(\overline t)
-\frac{\Dak\Psi(1-\ak)}{\Gamma(1-\ak)}\int_{a_k}^{t_k}(t_k-\t)^{-\ak}[x[\t]_k-x[a_k]_k]d\t,
\end{equation}
\begin{multline*}
\PRCI x(\overline t)=\PRCIII x(\overline t)\\
+\frac{\Dak}{\Gamma(2-\ak)}\int_{t_k}^{b_k}(\t-t_k)^{1-\ak}
\frac{\partial x}{\partial t_k}[\t]_k\left[\frac{1}{1-\ak}-\ln(\t-t_k)\right]d\t
\end{multline*}
and
\begin{equation*}
\PRCI x(\overline t)=\PRCII x(\overline t)
+\frac{\Dak\Psi(1-\ak)}{\Gamma(1-\ak)}
\int_{t_k}^{b_k}(\t-t_k)^{-\ak}[x[\t]_k-x[b_k]_k]d\t.
\end{equation*}
\end{Theorem}


\section{Approximation of variable order Caputo derivatives}
\label{sec:theorems}

Let $p\in\mathbb{N}$. We define
\begin{equation*}
\begin{split}
A_p &=\DS \frac{1}{\Gamma(p+1-\ak)}\left[1+\sum_{l=n-p+1}^N
\frac{\Gamma(\ak-n+l)}{\Gamma(\ak-p)(l-n+p)!}  \right],\\
B_p &=  \DS\frac{\Gamma(\ak-n+p)}{\Gamma(1-\ak)\Gamma(\ak)(p-n)!},\\
V_p(\overline t) &= \DS\int_{a_k}^{t_k}(\t-a_k)^{p-n}\frac{\partial x}{\partial t_k}[\t]_kd\t,\\
L_{p}(\overline t) &=\DS \max_{\t\in[a_k,t_k]}\left| \frac{\partial^{p} x}{\partial t_k^{p}}[\t]_k \right|.
\end{split}
\end{equation*}

\begin{Theorem}
\label{teo1}
Let $x\in C^{n+1}\left(\prod_{i=1}^m[a_i,b_i],\mathbb{R}\right)$ with $n\in\mathbb{N}$.
Then, for all $k\in\{1,\ldots,m\}$ and for all $N \in \mathbb{N}$ such that $N \geq n$, we have
$$
\PLCIII x(\overline t) =\DS\sum_{p=1}^{n}A_p (t_k-a_k)^{p-\ak}
\frac{\partial^p x}{\partial t_k^p}[t_k]_k \DS
+\sum_{p=n}^N B_p (t_k-a_k)^{n-p-\ak} V_p(\overline t)+E(\overline t).
$$
The approximation error $E(\overline t)$ is bounded by
$$
E(\overline t)\leq L_{n+1}(\overline t)
\frac{\exp((n-\ak)^2+n-\ak)}{\Gamma(n+1-\ak)N^{n-\ak}(n-\ak)}(t_k-a_k)^{n+1-\ak}.
$$
\end{Theorem}

\begin{proof}
By definition,
$$
\PLCIII x(\overline t)=\DS\frac{1}{\Gamma(1-\ak)}
\int_{a_k}^{t_k}(t_k-\t)^{-\ak}\frac{\partial x}{\partial t_k}[\t]_kd\t
$$
and, integrating by parts with $u'(\t)=(t_k-\t)^{-\ak}$
and $v(\t)=\frac{\partial x}{\partial t_k}[\t]_k$, we deduce that
$$
\PLCIII x(\overline t)=\DS\frac{(t_k-a_k)^{1-\ak}}{\Gamma(2-\ak)}
\frac{\partial x}{\partial t_k}[a_k]_k+\frac{1}{\Gamma(2-\ak)}
\int_{a_k}^{t_k}(t_k-\t)^{1-\ak}\frac{\partial^2 x}{\partial t_k^2}[\t]_kd\t.
$$
Integrating again by parts, taking $u'(\t)=(t_k-\t)^{1-\ak}$
and $v(\t)=\frac{\partial^2 x}{\partial t_k^2}[\t]_k$, we get
\begin{multline*}
\PLCIII x(\overline t)
=\DS\frac{(t_k-a_k)^{1-\ak}}{\Gamma(2-\ak)}\frac{\partial x}{\partial t_k}[a_k]_k
+\frac{(t_k-a_k)^{2-\ak}}{\Gamma(3-\ak)}\frac{\partial^2 x}{\partial t_k^2}[a_k]_k\\
\DS +\frac{1}{\Gamma(3-\ak)}\int_{a_k}^{t_k}(t_k-\t)^{2-\ak}
\frac{\partial^3 x}{\partial t_k^3}[\t]_kd\t.
\end{multline*}
Repeating the same procedure $n-2$ more times, we get the expansion formula
\begin{multline*}
\PLCIII x(\overline t) = \DS\sum_{p=1}^n \frac{(t_k-a_k)^{p-\ak}}{\Gamma(p+1-\ak)}
\frac{\partial^p x}{\partial t_k^p}[a_k]_k\\
\DS +\frac{1}{\Gamma(n+1-\ak)}\int_{a_k}^{t_k}(t_k-\t)^{n-\ak}
\frac{\partial^{n+1} x}{\partial t_k^{n+1}}[\t]_kd\t.
\end{multline*}
Using the equalities
\begin{equation*}
\begin{split}
(t_k-\t)^{n-\ak}&=\DS(t_k-a_k)^{n-\ak}\left(
1-\frac{\t-a_k}{t_k-a_k}\right)^{n-\ak}\\
&=\DS(t_k-a_k)^{n-\ak}\left[\sum_{p=0}^N \C (-1)^p
\frac{(\t-a_k)^p}{(t_k-a_k)^p}+\overline E(\overline t)\right]
\end{split}
\end{equation*}
with
$$
\overline E(\overline t)
=\sum_{p=N+1}^\infty \C (-1)^p \frac{(\t-a_k)^p}{(t_k-a_k)^p},
$$
we arrive at
\begin{equation*}
\begin{split}
\PLCIII x(\overline t)
&=\DS\sum_{p=1}^n \frac{(t_k-a_k)^{p-\ak}}{\Gamma(p+1-\ak)}
\frac{\partial^p x}{\partial t_k^p}[a_k]_k\\
& \quad \DS +\frac{(t_k-a_k)^{n-\ak}}{\Gamma(n+1-\ak)}\int_{a_k}^{t_k}
\sum_{p=0}^N \C (-1)^p \frac{(\t-a_k)^p}{(t_k-a_k)^p}
\frac{\partial^{n+1} x}{\partial t_k^{n+1}}[\t]_kd\t+E(\overline t)\\
& =\DS\sum_{p=1}^n \frac{(t_k-a_k)^{p-\ak}}{\Gamma(p+1-\ak)}
\frac{\partial^p x}{\partial t_k^p}[a_k]_k\\
& \quad \DS +\frac{(t_k-a_k)^{n-\ak}}{\Gamma(n+1-\ak)}\sum_{p=0}^N \C
\frac{(-1)^p}{(t_k-a_k)^p}\int_{a_k}^{t_k}(\t-a_k)^p
\frac{\partial^{n+1} x}{\partial t_k^{n+1}}[\t]_kd\t+E(\overline t)
\end{split}
\end{equation*}
with
$$
E(\overline t)=\frac{(t_k-a_k)^{n-\ak}}{\Gamma(n+1-\ak)}\int_{a_k}^{t_k}
\overline E(\overline t)\frac{\partial^{n+1} x}{\partial t_k^{n+1}}[\t]_kd\t.
$$
Now, we split the last sum into $p=0$ and the remaining terms $p=1,\ldots,N$
and integrate by parts with $u(\t)=(\t-a_k)^p$ and
$v'(\t)=\frac{\partial^{n+1} x}{\partial t_k^{n+1}}[\t]_k$.
Observing that
$$
\C(-1)^p=\frac{\Gamma(\ak-n+p)}{\Gamma(\ak-n)p!},
$$
we obtain:
\begin{equation*}
\begin{split}
&\frac{(t_k-a_k)^{n-\ak}}{\Gamma(n+1-\ak)}\sum_{p=0}^N \C \frac{(-1)^p}{(t_k-a_k)^p}
\int_{a_k}^{t_k}(\t-a_k)^p\frac{\partial^{n+1} x}{\partial t_k^{n+1}}[\t]_kd\t\\
&=\frac{(t_k-a_k)^{n-\ak}}{\Gamma(n+1-\ak)}\left[
\frac{\partial^n x}{\partial t_k^n}[t_k]_k-\frac{\partial^n x}{\partial t_k^n}[a_k]_k\right]
+\frac{(t_k-a_k)^{n-\ak}}{\Gamma(n+1-\ak)}\sum_{p=1}^N
\frac{\Gamma(\ak-n+p)}{\Gamma(\ak-n)p!(t_k-a_k)^{p}}\\
&\quad \times\left[(t_k-a_k)^p \frac{\partial^n x}{\partial t_k^n}[t_k]_k
-\int_{a_k}^{t_k}p(\t-a_k)^{p-1}\frac{\partial^n x}{\partial t_k^n}[\t]_kd\t\right]
\end{split}
\end{equation*}
\begin{equation*}
\begin{split}
&=-\frac{(t_k-a_k)^{n-\ak}}{\Gamma(n+1-\ak)}\frac{\partial^n x}{\partial t_k^n}[a_k]_k
+\frac{(t_k-a_k)^{n-\ak}}{\Gamma(n+1-\ak)}\frac{\partial^n x}{\partial t_k^n}[t_k]_k\left[
1+\sum_{p=1}^N \frac{\Gamma(\ak-n+p)}{\Gamma(\ak-n)p!}\right]\\
& \quad +\frac{(t_k-a_k)^{n-\ak-1}}{\Gamma(n-\ak)}\sum_{p=1}^N
\frac{\Gamma(\ak-n+p)}{\Gamma(\ak+1-n)(p-1)!(t_k-a_k)^{p-1}}
\int_{a_k}^{t_k}(\t-a_k)^{p-1}\frac{\partial^{n} x}{\partial t_k^{n}}[\t]_kd\t.
\end{split}
\end{equation*}
Thus, we get
\begin{equation*}
\begin{split}
\PLCIII x(\overline t)& =\DS\sum_{p=1}^{n-1}
\frac{(t_k-a_k)^{p-\ak}}{\Gamma(p+1-\ak)}\frac{\partial^p x}{\partial t_k^p}[a_k]_k\\
& \qquad \DS +\frac{(t_k-a_k)^{n-\ak}}{\Gamma(n+1-\ak)}
\frac{\partial^n x}{\partial t_k^n}[t_k]_k\left[1+\sum_{p=1}^N
\frac{\Gamma(\ak-n+p)}{\Gamma(\ak-n)p!}  \right]\\
& \qquad \DS +\frac{(t_k-a_k)^{n-\ak-1}}{\Gamma(n-\ak)}\sum_{p=1}^N
\frac{\Gamma(\ak-n+p)}{\Gamma(\ak+1-n)(p-1)!(t_k-a_k)^{p-1}} \\
& \qquad \DS\times\int_{a_k}^{t_k}(\t-a_k)^{p-1}
\frac{\partial^{n} x}{\partial t_k^{n}}[\t]_kd\t+E(\overline t).
\end{split}
\end{equation*}
Repeating the process $n-1$ more times with respect to the last sum, that is,
splitting the first term of the sum and integrating by parts the obtained result,
we arrive to
\begin{equation*}
\begin{split}
\PLCIII x(\overline t)& =\DS\sum_{p=1}^{n}
\frac{(t_k-a_k)^{p-\ak}}{\Gamma(p+1-\ak)}\frac{\partial^p x}{\partial t_k^p}[t_k]_k\left[
1+\sum_{l=n-p+1}^N \frac{\Gamma(\ak-n+l)}{\Gamma(\ak-p)(l-n+p)!}  \right]\\
& \quad \DS +\sum_{p=n}^N \frac{\Gamma(\ak-n+p)}{\Gamma(1-\ak)\Gamma(\ak)(p-n)!}(t_k-a_k)^{n-p-\ak} \\
& \quad \DS\times\int_{a_k}^{t_k}(\t-a_k)^{p-n}\frac{\partial x}{\partial t_k}[\t]_kd\t+E(\overline t).
\end{split}
\end{equation*}
We now seek the upper bound formula for $E(\overline t)$.
Using the two relations
$$
\left|  \frac{\t-a_k}{t_k-a_k}\right|\leq 1,
\, \mbox{ if } \, \t\in[a_k,t_k]
\quad \mbox{ and } \quad
\left| \C \right|\leq \frac{\exp((n-\ak)^2+n-\ak)}{p^{n+1-\ak}},
$$
we get
\begin{equation*}
\begin{split}
\overline E(\overline t)& \leq\DS \sum_{p=N+1}^\infty
\frac{\exp((n-\ak)^2+n-\ak)}{p^{n+1-\ak}} \\
& \DS\leq \int_N^\infty \frac{\exp((n-\ak)^2+n-\ak)}{p^{n+1-\ak}}\, dp
=\frac{\exp((n-\ak)^2+n-\ak)}{N^{n-\ak}(n-\ak)}.
\end{split}
\end{equation*}
Then,
$$
E(\overline t)\leq L_{n+1}(\overline t)
\frac{\exp((n-\ak)^2+n-\ak)}{\Gamma(n+1-\ak)N^{n-\ak}(n-\ak)}(t_k-a_k)^{n+1-\ak}.
$$
This concludes the proof.
\end{proof}

\begin{Remark}
In Theorem~\ref{teo1} we have
$$
\lim_{N\to\infty}E(\overline t)=0
$$
for all $\overline t\in \prod_{i=1}^m[a_i,b_i]$ and $n\in\mathbb{N}$.
\end{Remark}

\begin{Theorem}
\label{teo2}
Let $x\in C^{n+1}\left(\prod_{i=1}^m[a_i,b_i],\mathbb{R}\right)$ with $n\in\mathbb{N}$.
Then, for all $k\in\{1,\ldots,m\}$ and for all $N \in \mathbb{N}$ such that $N \geq n$, we have
\begin{multline*}
\PLCI x(\overline t) =\DS\sum_{p=1}^{n}A_p (t_k-a_k)^{p-\ak}
\frac{\partial^p x}{\partial t_k^p}[t_k]_k \DS
+\sum_{p=n}^N B_p (t_k-a_k)^{n-p-\ak} V_p(\overline t)\\
+\frac{\Dak(t_k-a_k)^{1-\ak}}{\Gamma(2-\ak)}\left[\left(\frac{1}{1-\ak}
-\ln(t_k-a_k)\right)\sum_{p=0}^N\D\frac{(-1)^p}{(t_k-a_k)^{p}} V_{n+p}(\overline t)\right.\\
\left.+\sum_{p=0}^N\D(-1)^p\sum_{r=1}^N\frac{1}{r(t_k-a_k)^{p+r}}
V_{n+p+r}(\overline t)\right]+E(\overline t).
\end{multline*}
The approximation error $E(\overline t)$ is bounded by
\begin{equation*}
\begin{split}
E(\overline t)& \leq\DS L_{n+1}(\overline t)
\frac{\exp((n-\ak)^2+n-\ak)}{\Gamma(n+1-\ak)N^{n-\ak}(n-\ak)}(t_k-a_k)^{n+1-\ak}\\
&\quad \DS+\left|\Dak\right|L_1(\overline t)
\frac{{\exp((1-\ak)^2+1-\ak)}}{\Gamma(2-\ak)N^{1-\ak}(1-\ak)}\\
&\quad \DS \times \left[\left|\frac{1}{1-\ak}
-\ln(t_k-a_k)\right|+\frac{1}{N}\right](t_k-a_k)^{2-\ak}.
\end{split}
\end{equation*}
\end{Theorem}

\begin{proof}
Taking into account relation \eqref{relation} and Theorem~\ref{teo1},
we only need to expand the term
\begin{equation}
\label{relation2}
\frac{\Dak}{\Gamma(2-\ak)}\int_{a_k}^{t_k}(t_k-\t)^{1-\ak}
\frac{\partial x}{\partial t_k}[\t]_k\left[\frac{1}{1-\ak}-\ln(t_k-\t)\right]d\t.
\end{equation}
Splitting the integral, and using the expansion formulas
\begin{equation*}
\begin{split}
(t_k-\t)^{1-\ak}&=\DS(t_k-a_k)^{1-\ak}\left(1-\frac{\t-a_k}{t_k-a_k}\right)^{1-\ak}\\
&=\DS(t_k-a_k)^{1-\ak}\left[\sum_{p=0}^N \D (-1)^p
\frac{(\t-a_k)^p}{(t_k-a_k)^p}+\overline E_1(\overline t)\right]
\end{split}
\end{equation*}
with
$$
\overline E_1(\overline t)
=\sum_{p=N+1}^\infty \D (-1)^p \frac{(\t-a_k)^p}{(t_k-a_k)^p}
$$
and
\begin{equation*}
\begin{split}
\ln(t_k-\t)&=\DS\ln(t_k-a_k)+\ln\left(1-\frac{\t-a_k}{t_k-a_k}\right)\\
&=\DS\ln(t_k-a_k)-\sum_{r=1}^N \frac{1}{r}
\frac{(\t-a_k)^r}{(t_k-a_k)^r}-\overline E_2(\overline t)
\end{split}
\end{equation*}
with
$$
\overline E_2(\overline t)=\sum_{r=N+1}^\infty
\frac{1}{r} \frac{(\t-a_k)^r}{(t_k-a_k)^r},
$$
we conclude that \eqref{relation2} is equivalent to
\begin{equation*}
\begin{split}
&\frac{\Dak}{\Gamma(2-\ak)}\left[\left(\frac{1}{1-\ak}-\ln(t_k-a_k)\right)
\int_{a_k}^{t_k}(t_k-\t)^{1-\ak}\frac{\partial x}{\partial t_k}[\t]_kd\t\right.\\
&\qquad \left. -\int_{a_k}^{t_k}(t_k-\t)^{1-\ak} \ln\left(1-\frac{\t-a_k}{t_k-a_k}\right)
\frac{\partial x}{\partial t_k}[\t]_kd\t\right]\\
&=\frac{\Dak}{\Gamma(2-\ak)}\left[
\left(\frac{1}{1-\ak}-\ln(t_k-a_k)\right)\right.\\
&\qquad \times\int_{a_k}^{t_k}(t_k-a_k)^{1-\ak}\sum_{p=0}^N \D (-1)^p
\frac{(\t-a_k)^p}{(t_k-a_k)^p}\frac{\partial x}{\partial t_k}[\t]_kd\t\\
&\qquad\left. +\int_{a_k}^{t_k}(t_k-a_k)^{1-\ak}
\sum_{p=0}^N \D (-1)^p \frac{(\t-a_k)^p}{(t_k-a_k)^p}\sum_{r=1}^N \frac{1}{r}
\frac{(\t-a_k)^r}{(t_k-a_k)^r}\frac{\partial x}{\partial t_k}[\t]_kd\t  \right]\\
&\qquad+\frac{\Dak}{\Gamma(2-\ak)}\left[\left(\frac{1}{1-\ak}-\ln(t_k-a_k)\right)
\int_{a_k}^{t_k}(t_k-a_k)^{1-\ak}\overline E_1(\overline t)
\frac{\partial x}{\partial t_k}[\t]_kd\t\right.\\
&\qquad\left. +\int_{a_k}^{t_k}(t_k-a_k)^{1-\ak}\overline E_1(\overline t)
\overline E_2(\overline t)\frac{\partial x}{\partial t_k}[\t]_kd\t \right]\\
&=\frac{\Dak(t_k-a_k)^{1-\ak}}{\Gamma(2-\ak)}\left[\left(\frac{1}{1-\ak}
-\ln(t_k-a_k)\right)\sum_{p=0}^N\D\frac{(-1)^p}{(t_k-a_k)^{p}} V_{n+p}(\overline t)\right.\\
&\qquad\left.+\sum_{p=0}^N\D(-1)^p\sum_{r=1}^N\frac{1}{r(t_k-a_k)^{p+r}}
V_{n+p+r}(\overline t)\right]+\frac{\Dak (t_k-a_k)^{1-\ak}}{\Gamma(2-\ak)}\\
&\qquad\times\left[\left(\frac{1}{1-\ak}
-\ln(t_k-a_k)\right)\int_{a_k}^{t_k}\overline E_1(\overline t)
\frac{\partial x}{\partial t_k}[\t]_kd\t
+\int_{a_k}^{t_k}\overline E_1(\overline t)\overline E_2(\overline t)
\frac{\partial x}{\partial t_k}[\t]_kd\t  \right].
\end{split}
\end{equation*}
For the error analysis, we know from Theorem~\ref{teo1} that
$$
\overline E_1(\overline t)
\leq\frac{\exp((1-\ak)^2+1-\ak)}{N^{1-\ak}(1-\ak)}.
$$
Then,
\begin{equation}
\label{error1}
\left| \int_{a_k}^{t_k}(t_k-a_k)^{1-\ak}\overline E_1(\overline t)
\frac{\partial x}{\partial t_k}[\t]_kd\t \right|
\leq L_1(\overline t)\frac{{\exp((1-\ak)^2+1-\ak)}}{N^{1-\ak}(1-\ak)}(t_k-a_k)^{2-\ak}.
\end{equation}
On the other hand, we have
\begin{equation}
\label{error2}
\begin{split}
&\left|\int_{a_k}^{t_k}(t_k-a_k)^{1-\ak}\overline E_1(\overline t)\overline E_2(\overline t)
\frac{\partial x}{\partial t_k}[\t]_kd\t \right|\\
&\leq L_1(\overline t)\frac{{\exp((1-\ak)^2+1-\ak)}}{N^{1-\ak}(1-\ak)}(t_k-a_k)^{1-\ak}
\sum_{r=N+1}^\infty\frac{1}{r(t_k-a_k)^r}\int_{a_k}^{t_k}(\t-a_k)^rd\t\\
&= L_1(\overline t)\frac{{\exp((1-\ak)^2+1-\ak)}}{N^{1-\ak}(1-\ak)}(t_k-a_k)^{1-\ak}
\sum_{r=N+1}^\infty\frac{t_k-a_k}{r(r+1)}\\
&\leq L_1(\overline t)\frac{{\exp((1-\ak)^2+1-\ak)}}{N^{2-\ak}(1-\ak)}(t_k-a_k)^{2-\ak}.
\end{split}
\end{equation}
We get the desired result by combining inequalities \eqref{error1} and \eqref{error2}.
\end{proof}

\begin{Theorem}
\label{teo3}
Let $x\in C^{n+1}(\prod_{i=1}^m[a_i,b_i],\mathbb{R})$ with $n\in\mathbb{N}$. Then,
for all $k\in\{1,\ldots,m\}$ and for all $N \in \mathbb{N}$ such that $N \geq n$, we have
\begin{equation*}
\begin{split}
&\PLCII x(\overline t) =\DS\sum_{p=1}^{n}A_p (t_k-a_k)^{p-\ak}
\frac{\partial^p x}{\partial t_k^p}[t_k]_k \DS
+\sum_{p=n}^N B_p (t_k-a_k)^{n-p-\ak} V_p(\overline t)\\
&\qquad +\frac{\Dak(t_k-a_k)^{1-\ak}}{\Gamma(2-\ak)}\left[\left(\Psi(2-\ak)-\ln(t_k-a_k)\right)
\sum_{p=0}^N\D\frac{(-1)^p}{(t_k-a_k)^{p}} V_{n+p}(\overline t)\right.\\
&\qquad \left.+\sum_{p=0}^N\D(-1)^p\sum_{r=1}^N\frac{1}{r(t_k-a_k)^{p+r}}
V_{n+p+r}(\overline t)\right]+E(\overline t).
\end{split}
\end{equation*}
The approximation error $E(\overline t)$ is bounded by
\begin{equation*}
\begin{split}
E(\overline t)&\leq\DS L_{n+1}(\overline t)
\frac{\exp((n-\ak)^2+n-\ak)}{\Gamma(n+1-\ak)N^{n-\ak}(n-\ak)}(t_k-a_k)^{n+1-\ak}\\
& \quad \DS+\left|\Dak\right|L_1(\overline t)
\frac{{\exp((1-\ak)^2+1-\ak)}}{\Gamma(2-\ak)N^{1-\ak}(1-\ak)}\\
& \quad \DS\times \left[\left|\Psi(2-\ak)-\ln(t_k-a_k)\right|
+\frac{1}{N}\right](t_k-a_k)^{2-\ak}.
\end{split}
\end{equation*}
\end{Theorem}

\begin{proof}
Starting with relation \eqref{relation3},
and integrating by parts the integral, we obtain that
$$
\PLCII x(\overline t)=\PLCI x(\overline t)+\frac{\Dak\Psi(1-\ak)}{\Gamma(2-\ak)}
\int_{a_k}^{t_k}(t_k-\t)^{1-\ak}\frac{\partial x}{\partial t_k}[\t]_k d\t.
$$
The rest of the proof is similar to the one of Theorem~\ref{teo2}.
\end{proof}

\begin{Remark}
As particular cases of Theorems~\ref{teo1}, \ref{teo2} and \ref{teo3},
we obtain expansion formulas for $\LCI x(t)$, $\LCII x(t)$ and $\LCIII x(t)$.
\end{Remark}

With respect to the three right fractional operators of Definition~\ref{def:8},
we set, for $p\in\mathbb{N}$,
\begin{equation*}
\begin{split}
C_p &=\DS \frac{(-1)^{p}}{\Gamma(p+1-\ak)}\left[1+\sum_{l=n-p+1}^N
\frac{\Gamma(\ak-n+l)}{\Gamma(\ak-p)(l-n+p)!}  \right],\\
D_p &= \DS\frac{-\Gamma(\ak-n+p)}{\Gamma(1-\ak)\Gamma(\ak)(p-n)!},\\
W_p(\overline t)
&= \DS\int_{t_k}^{b_k}(b_k-\t)^{p-n}\frac{\partial x}{\partial t_k}[\t]_kd\t,\\
M_{p}(\overline t)
&=\DS \max_{\t\in[t_k,b_k]}\left| \frac{\partial^{p} x}{\partial t_k^{p}}[\t]_k \right|.
\end{split}
\end{equation*}
The expansion formulas are given in Theorems~\ref{thm:15}, \ref{thm:16} and \ref{thm:17}.
We omit the proofs since they are similar to the corresponding left ones.

\begin{Theorem}
\label{thm:15}
Let $x\in C^{n+1}\left(\prod_{i=1}^m[a_i,b_i],\mathbb{R}\right)$ with $n\in\mathbb{N}$.
Then, for all $k\in\{1,\ldots,m\}$ and for all $N \in \mathbb{N}$ such that $N \geq n$, we have
$$
\PRCIII x(\overline t) =\DS\sum_{p=1}^{n}C_p (b_k-t_k)^{p-\ak}
\frac{\partial^p x}{\partial t_k^p}[t_k]_k
\DS +\sum_{p=n}^N D_p (b_k-t_k)^{n-p-\ak} W_p(\overline t)+E(\overline t).
$$
The approximation error $E(\overline t)$ is bounded by
$$
E(\overline t)\leq M_{n+1}(\overline t)
\frac{\exp((n-\ak)^2+n-\ak)}{\Gamma(n+1-\ak)N^{n-\ak}(n-\ak)}(b_k-t_k)^{n+1-\ak}.
$$
\end{Theorem}

\begin{Theorem}
\label{thm:16}
Let $x\in C^{n+1}\left(\prod_{i=1}^m[a_i,b_i],\mathbb{R}\right)$ with $n\in\mathbb{N}$.
Then, for all $k\in\{1,\ldots,m\}$ and for all $N \in \mathbb{N}$ such that $N \geq n$, we have
\begin{equation*}
\begin{split}
&\PRCI x(\overline t) =\DS\sum_{p=1}^{n}C_p (b_k-t_k)^{p-\ak}
\frac{\partial^p x}{\partial t_k^p}[t_k]_k \DS
+\sum_{p=n}^N D_p (b_k-t_k)^{n-p-\ak} W_p(\overline t)\\
&\qquad +\frac{\Dak(b_k-t_k)^{1-\ak}}{\Gamma(2-\ak)}\left[\left(\frac{1}{1-\ak}
-\ln(b_k-t_k)\right)\sum_{p=0}^N\D\frac{(-1)^p}{(b_k-t_k)^{p}} W_{n+p}(\overline t)\right.\\
&\qquad \left.+\sum_{p=0}^N\D(-1)^p\sum_{r=1}^N\frac{1}{r(b_k-t_k)^{p+r}}
W_{n+p+r}(\overline t)\right]+E(\overline t).
\end{split}
\end{equation*}
The approximation error $E(\overline t)$ is bounded by
\begin{equation*}
\begin{split}
E(\overline t)&\leq M_{n+1}(\overline t)
\frac{\exp((n-\ak)^2+n-\ak)}{\Gamma(n+1-\ak)N^{n-\ak}(n-\ak)}(b_k-t_k)^{n+1-\ak}\\
&\quad \DS +\left|\Dak\right|M_1(\overline t)
\frac{{\exp((1-\ak)^2+1-\ak)}}{\Gamma(2-\ak)N^{1-\ak}(1-\ak)}\\
&\quad \DS \times \left[\left|\frac{1}{1-\ak}-\ln(b_k-t_k)\right|
+\frac{1}{N}\right](b_k-t_k)^{2-\ak}.
\end{split}
\end{equation*}
\end{Theorem}

\begin{Theorem}
\label{thm:17}
Let $x\in C^{n+1}\left(\prod_{i=1}^m[a_i,b_i],\mathbb{R}\right)$ with $n\in\mathbb{N}$.
Then, for all $k\in\{1,\ldots,m\}$ and for all $N \in \mathbb{N}$ such that $N \geq n$, we have
\begin{equation*}
\begin{split}
&\PRCI x(\overline t) =\DS\sum_{p=1}^{n}C_p (b_k-t_k)^{p-\ak}
\frac{\partial^p x}{\partial t_k^p}[t_k]_k \DS
+\sum_{p=n}^N D_p (b_k-t_k)^{n-p-\ak} W_p(\overline t)\\
&\qquad +\frac{\Dak(b_k-t_k)^{1-\ak}}{\Gamma(2-\ak)}\left[\left(\Psi(2-\ak)
-\ln(b_k-t_k)\right)\sum_{p=0}^N\D\frac{(-1)^p}{(b_k-t_k)^{p}}
W_{n+p}(\overline t)\right.\\
&\qquad \left.+\sum_{p=0}^N\D(-1)^p\sum_{r=1}^N\frac{1}{r(b_k-t_k)^{p+r}}
W_{n+p+r}(\overline t)\right]+E(\overline t).
\end{split}
\end{equation*}
The approximation error $E(\overline t)$ is bounded by
\begin{equation*}
\begin{split}
E(\overline t)&\leq M_{n+1}(\overline t)
\frac{\exp((n-\ak)^2+n-\ak)}{\Gamma(n+1-\ak)N^{n-\ak}(n-\ak)}(b_k-t_k)^{n+1-\ak}\\
&\quad \DS +\left|\Dak\right|M_1(\overline t)\frac{{\exp((1-\ak)^2+1-\ak)}}{\Gamma(2-\ak)N^{1-\ak}(1-\ak)}\\
&\quad \DS \times \left[\left|\Psi(2-\ak)-\ln(b_k-t_k)\right|+\frac{1}{N}\right](b_k-t_k)^{2-\ak}.
\end{split}
\end{equation*}
\end{Theorem}


\section{An example}
\label{sec:EX}

To test the accuracy of the proposed method, we compare the fractional derivative
of a concrete given function with some numerical approximations of it.
For $t\in[0,1]$, let $x(t)=t^2$ be the test function. For the order
of the fractional derivatives we consider two cases:
$$
\a=\frac{50t+49}{100} \quad \mbox{and} \quad \beta(t)=\frac{t+5}{10}.
$$
We consider the approximations given in Theorems~\ref{teo1}, \ref{teo2} and \ref{teo3},
with a fixed $n=1$ and $N\in\{2,4,6\}$. The error of approximating
$f(t)$ by $\tilde{f}(t)$ is measured by $|f(t)-\tilde{f}(t)|$.
See Figures~\ref{IntExp2}--\ref{IntExp7}.
\begin{figure}[!ht]
\begin{center}
\subfigure[${^C_0\mathbb{D}_t^{\a}} x(t)$]{\includegraphics[scale=0.3]{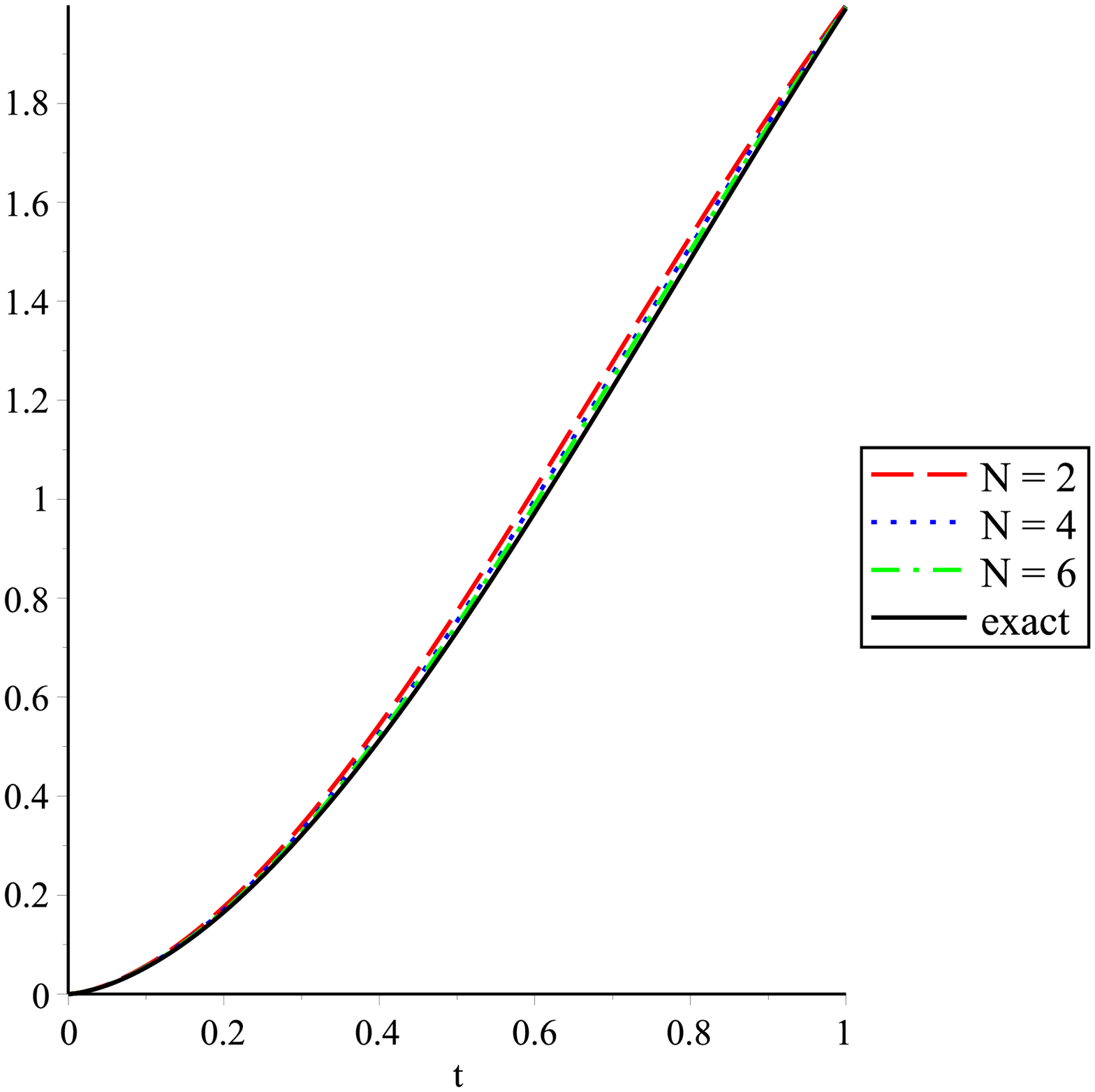}}\hspace{1cm}
\subfigure[Error]{\includegraphics[scale=0.3]{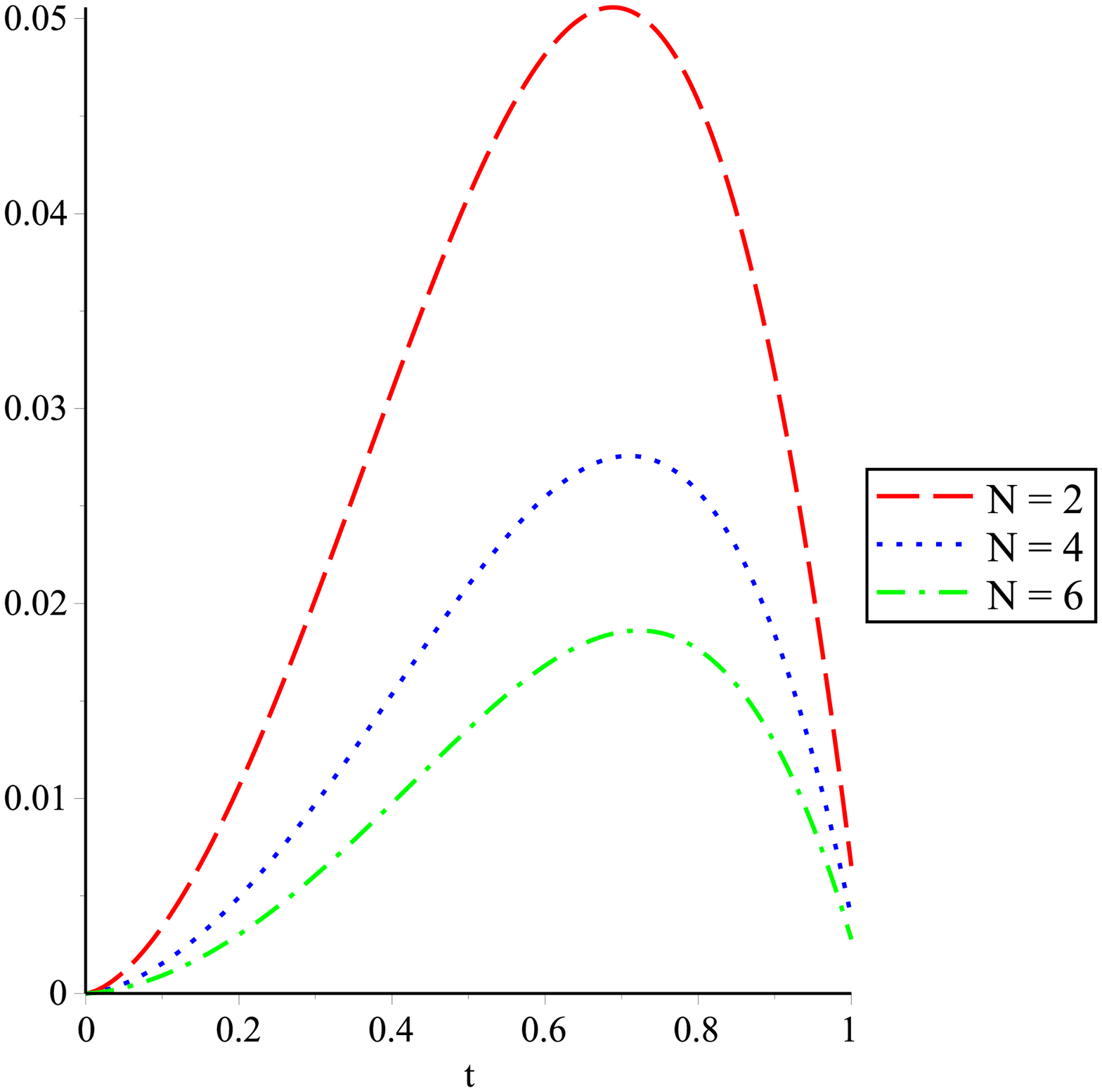}}
\end{center}
\caption{Type III left Caputo derivative of order $\a$
for the example of Section~\ref{sec:EX}---analytic
versus numerical approximations obtained from Theorem~\ref{teo1}.}
\label{IntExp2}
\end{figure}
\begin{figure}[!ht]
\begin{center}
\subfigure[${^C_0D_t^{\a}} x(t)$]{\includegraphics[scale=0.3]{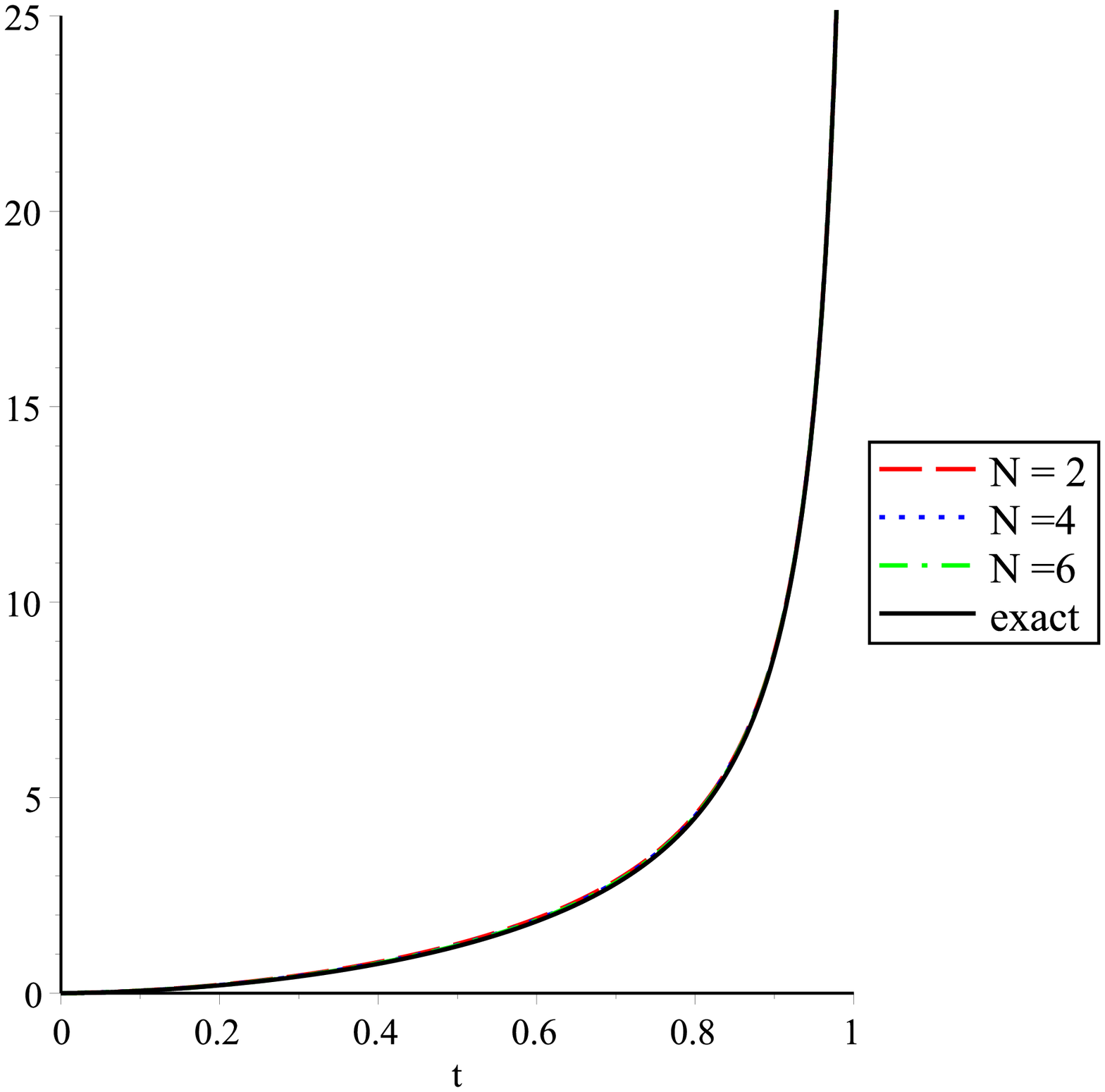}}\hspace{1cm}
\subfigure[Error]{\includegraphics[scale=0.3]{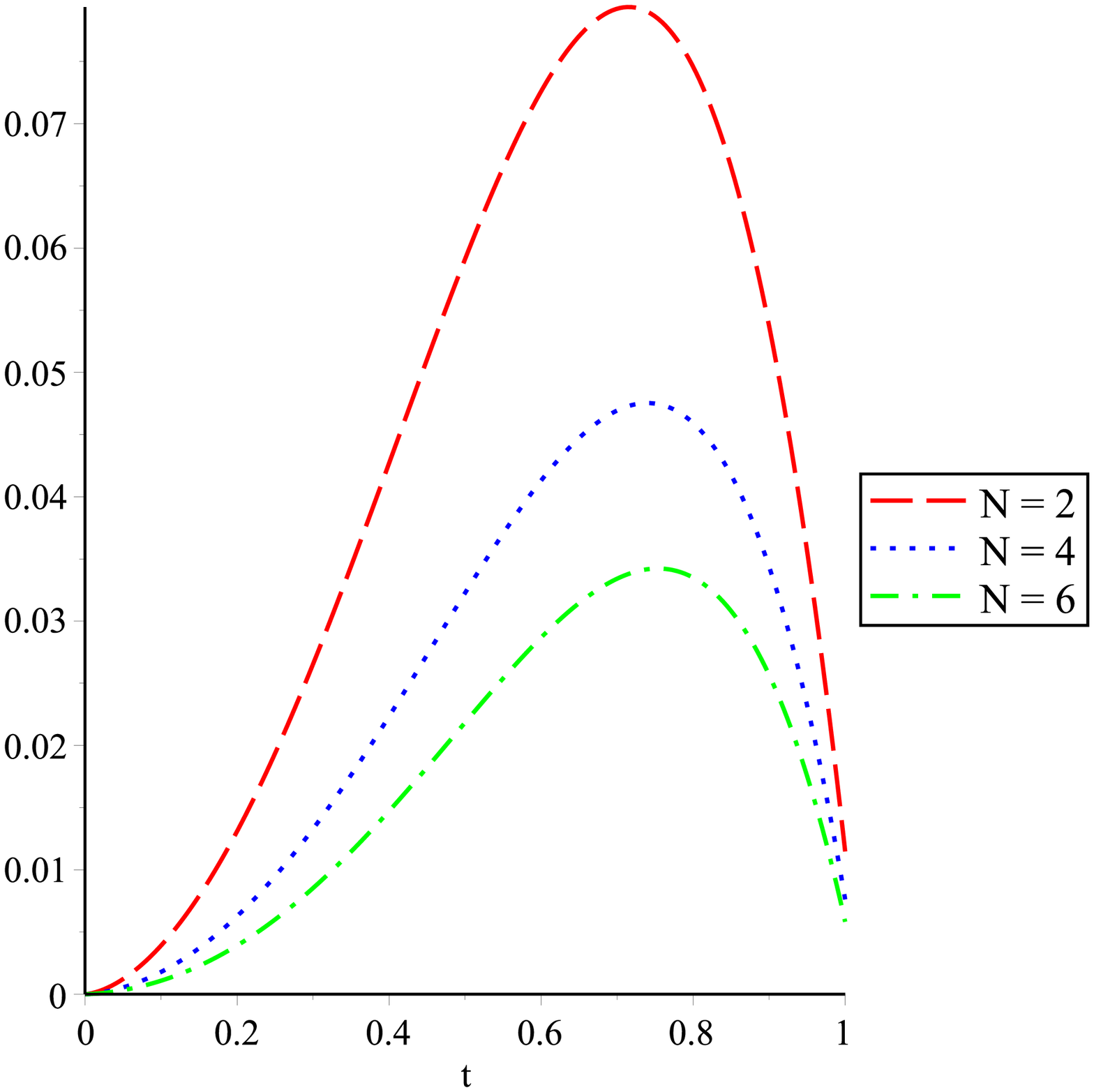}}
\end{center}
\caption{Type I left Caputo derivative of order $\a$
for the example of Section~\ref{sec:EX}---analytic
versus numerical approximations obtained from Theorem~\ref{teo2}.}
\label{IntExp3}
\end{figure}
\begin{figure}[!ht]
\begin{center}
\subfigure[${^C_0\mathcal{D}_t^{\a}} x(t)$]{\includegraphics[scale=0.3]{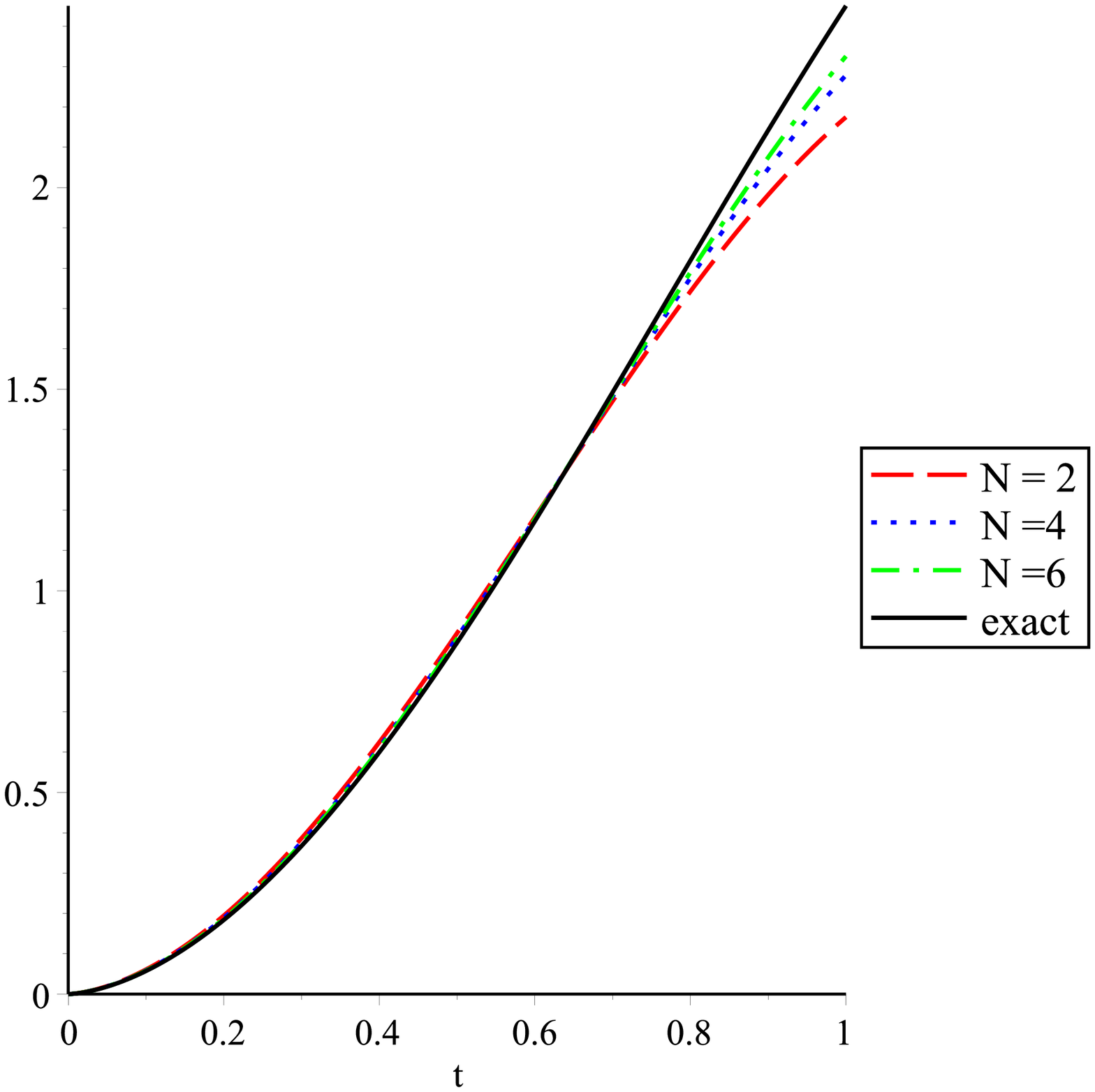}} \hspace{1cm}
\subfigure[Error]{\includegraphics[scale=0.3]{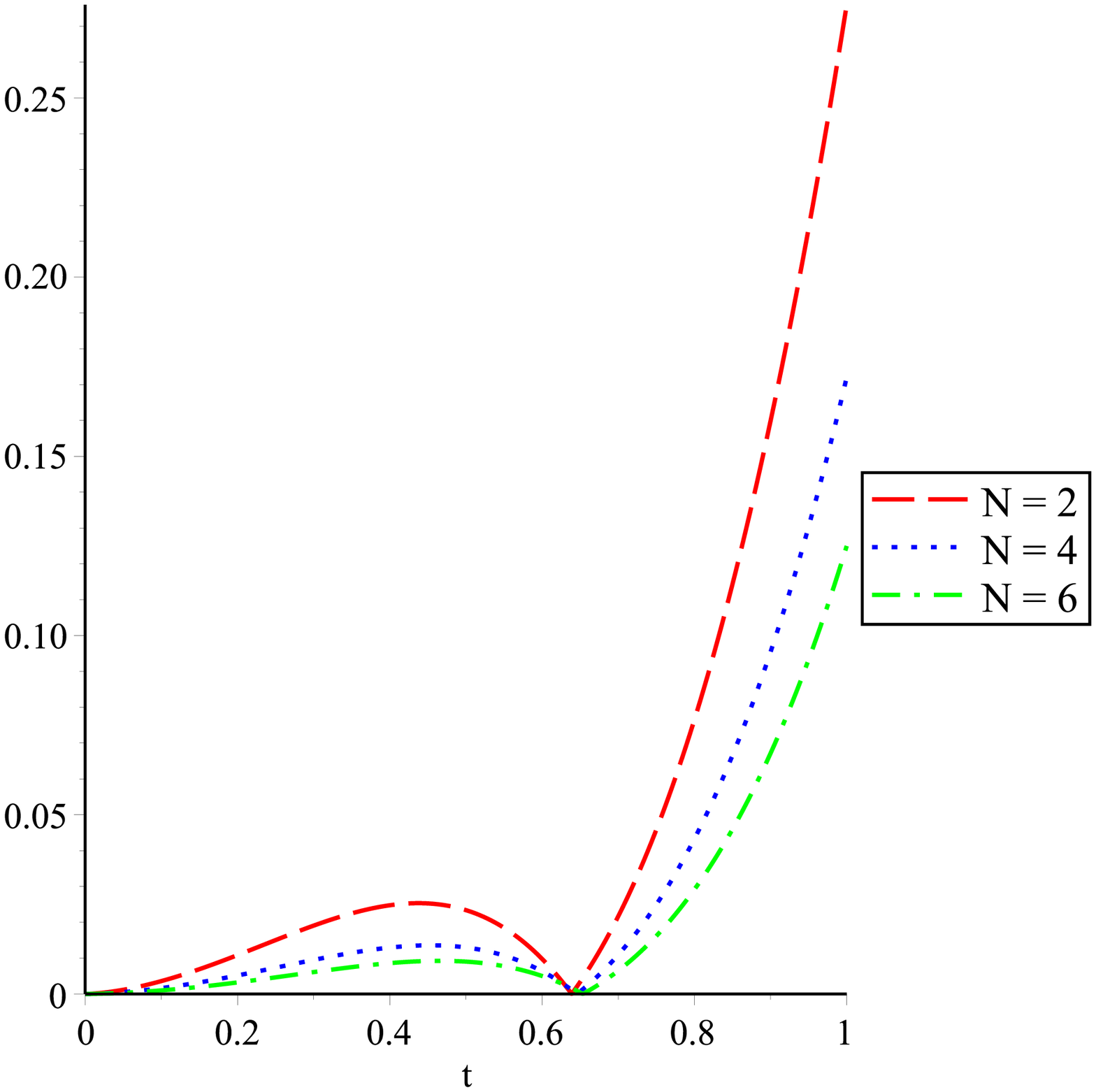}}
\end{center}
\caption{Type II left Caputo derivative of order $\a$
for the example of Section~\ref{sec:EX}---analytic
versus numerical approximations obtained from Theorem~\ref{teo3}.}
\label{IntExp4}
\end{figure}
\begin{figure}[!ht]
\begin{center}
\subfigure[${^C_0\mathbb{D}_t^{\beta(t)}} x(t)$]{\includegraphics[scale=0.3]{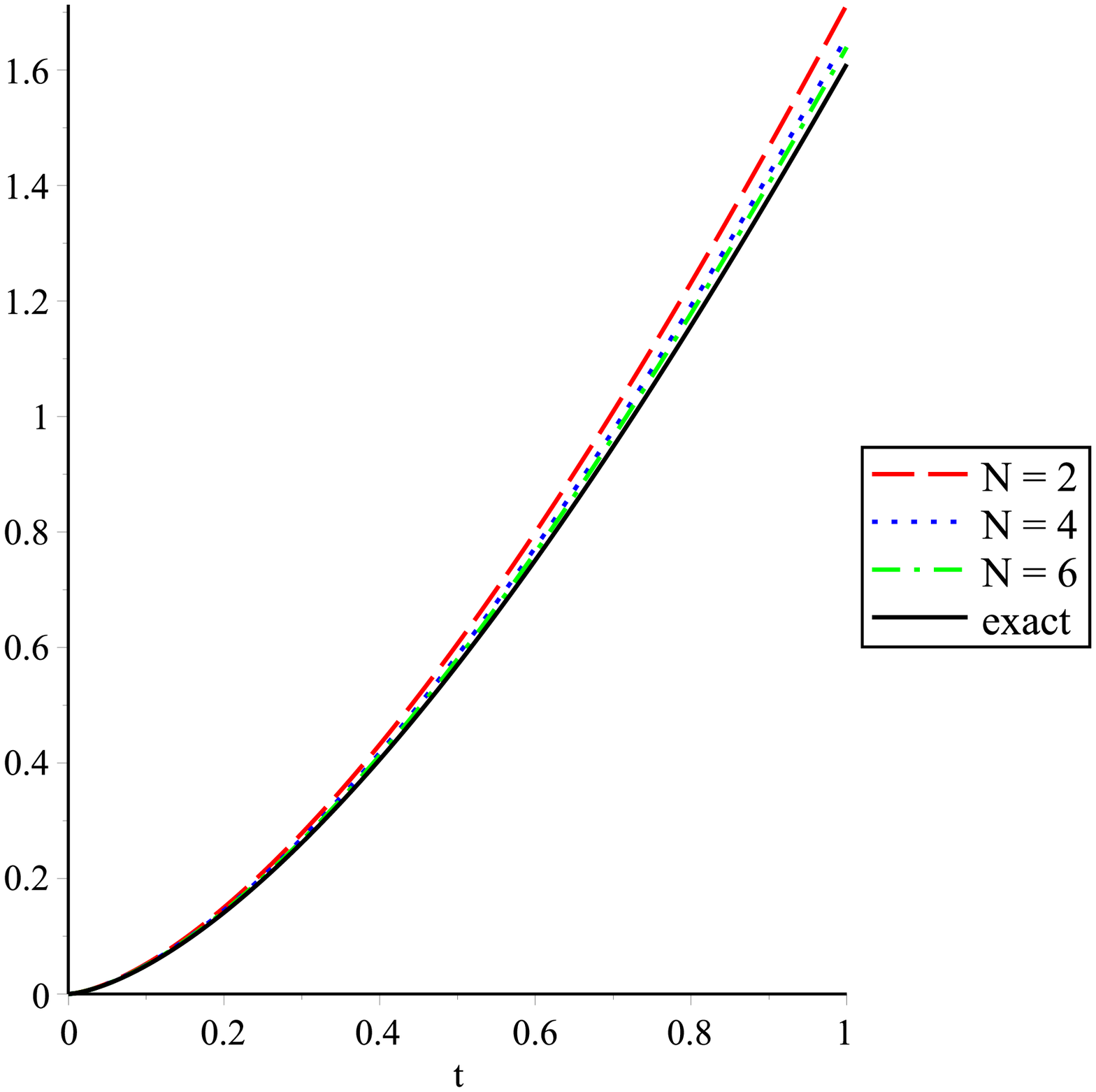}}\hspace{1cm}
\subfigure[Error]{\includegraphics[scale=0.3]{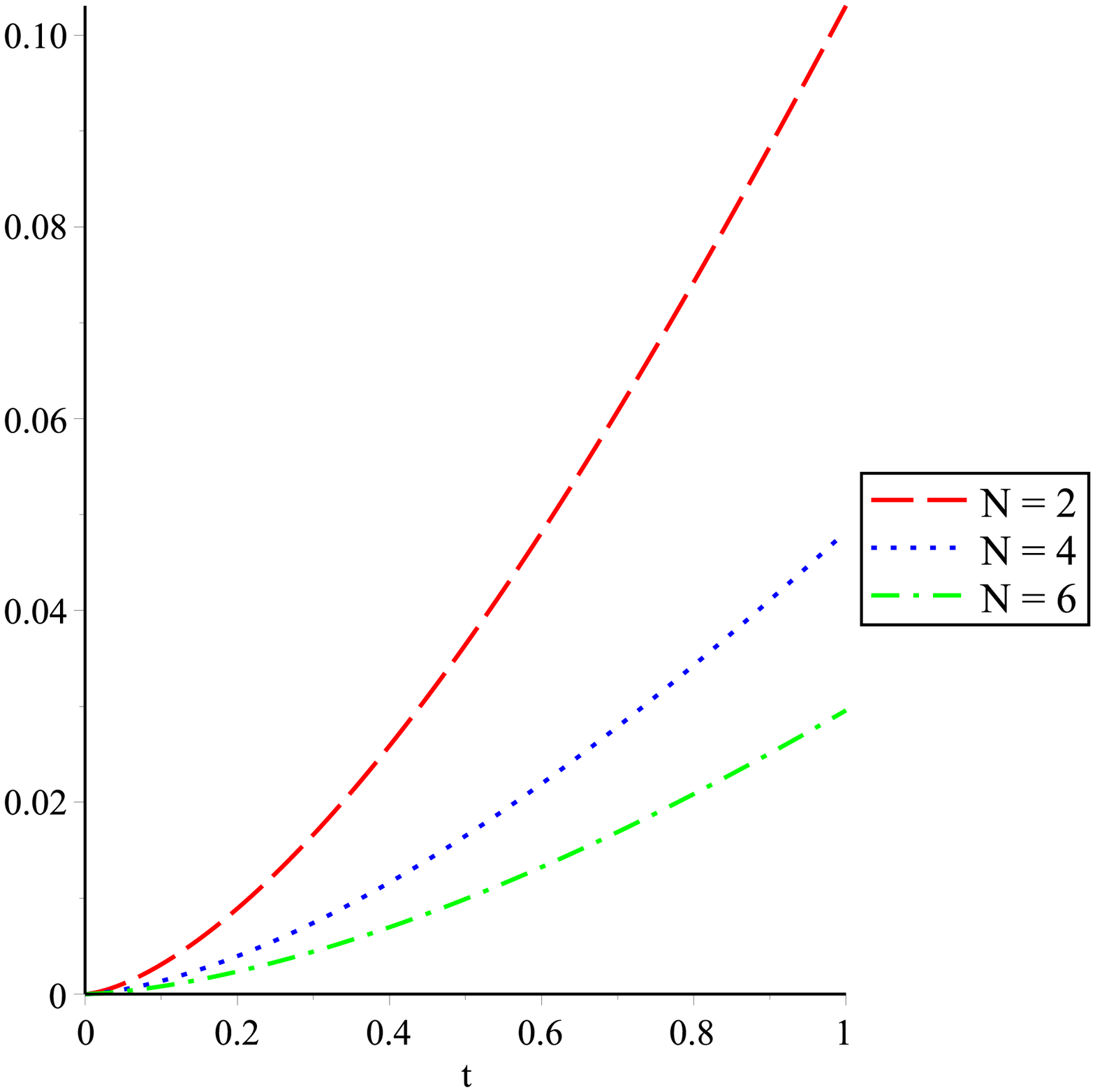}}
\end{center}
\caption{Type III left Caputo derivative of order $\beta(t)$
for the example of Section~\ref{sec:EX}---analytic
versus numerical approximations obtained from Theorem~\ref{teo1}.}
\label{IntExp5}
\end{figure}
\begin{figure}[!ht]
\begin{center}
\subfigure[${^C_0D_t^{\beta(t)}} x(t)$]{\includegraphics[scale=0.3]{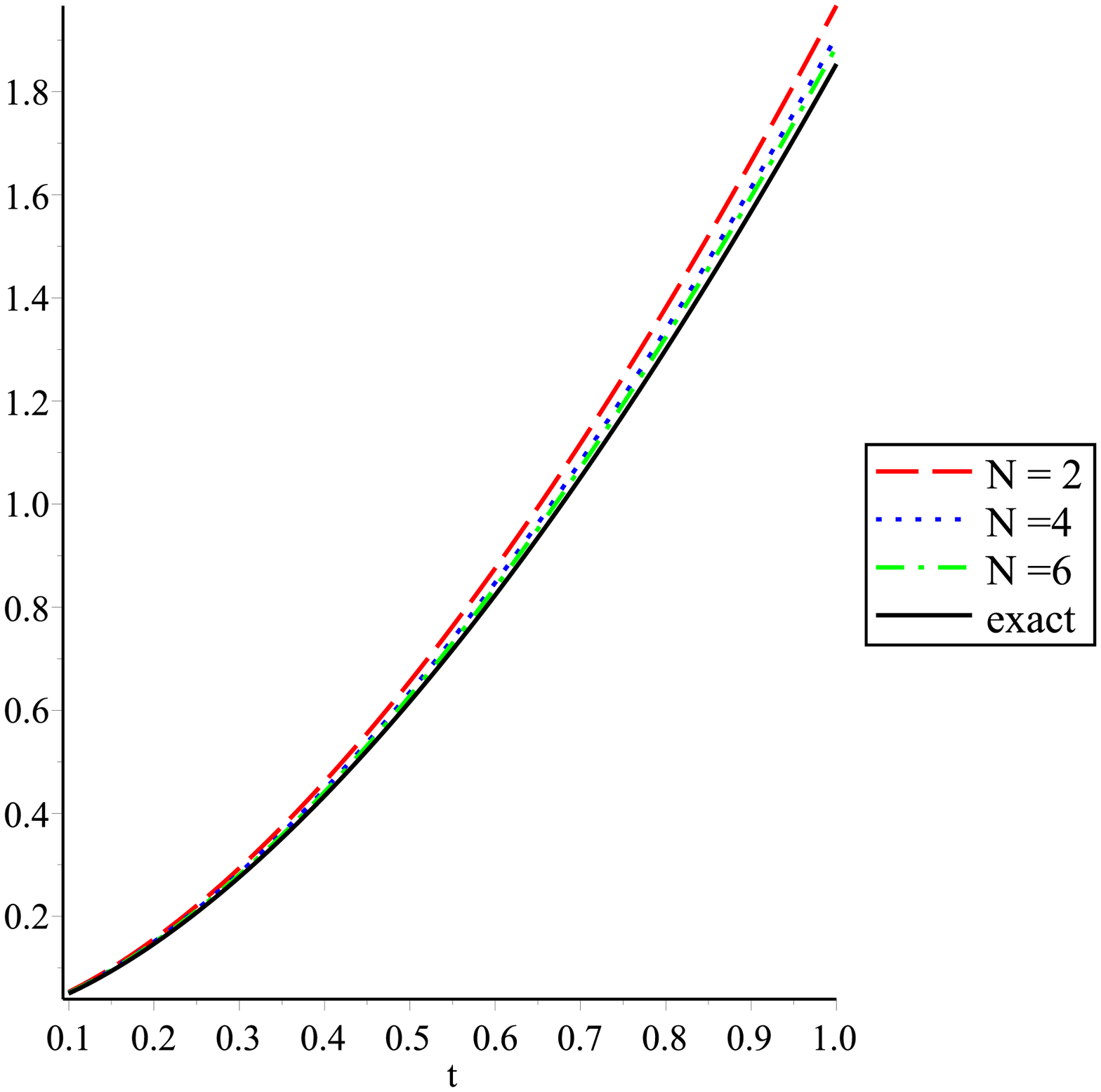}}\hspace{1cm}
\subfigure[Error]{\includegraphics[scale=0.3]{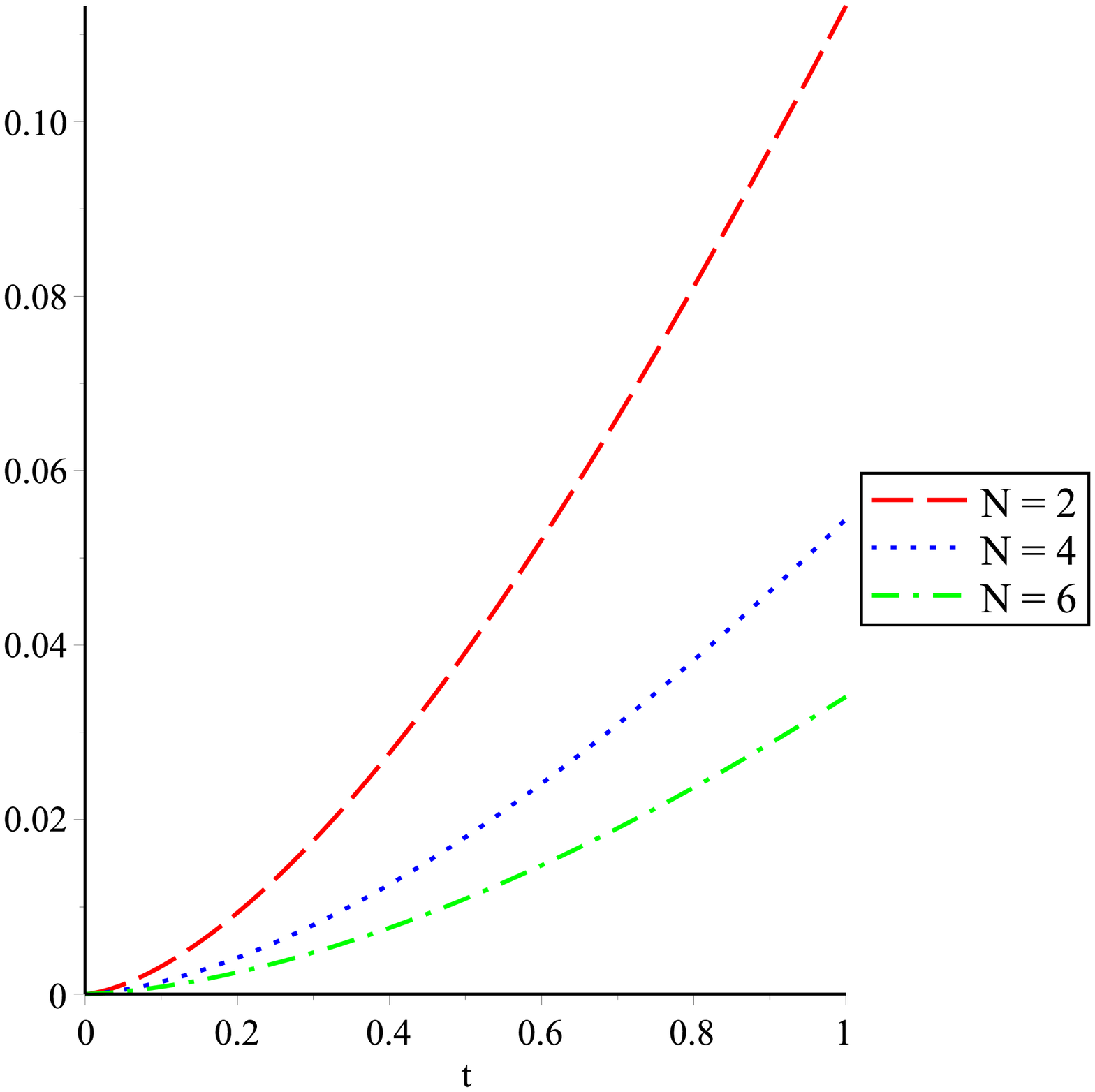}}
\end{center}
\caption{Type I left Caputo derivative of order $\beta(t)$
for the example of Section~\ref{sec:EX}---analytic
versus numerical approximations obtained from Theorem~\ref{teo2}.}
\label{IntExp6}
\end{figure}
\begin{figure}[!ht]
\begin{center}
\subfigure[${^C_0\mathcal{D}_t^{\beta(t)}} x(t)$]{\includegraphics[scale=0.3]{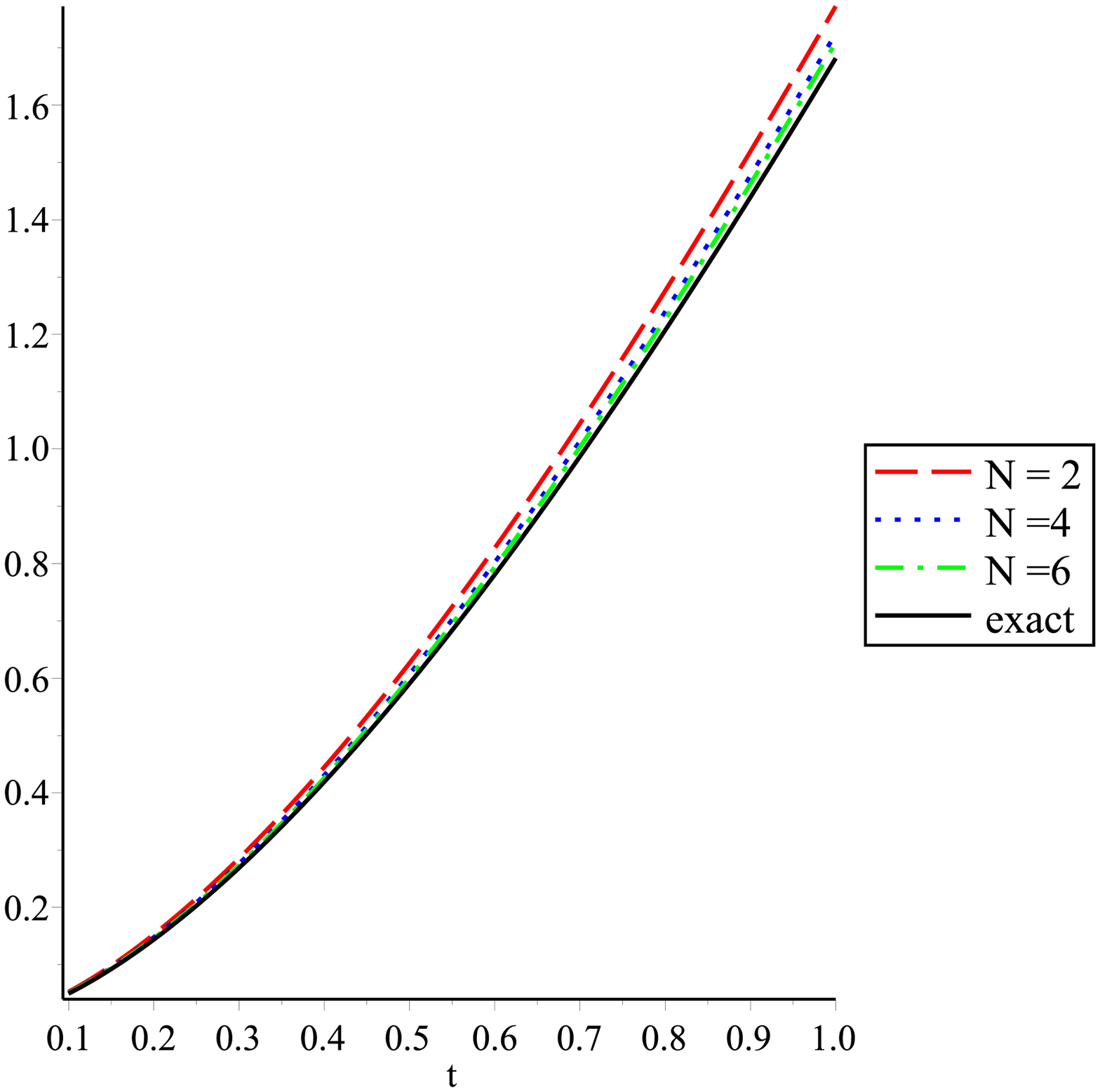}} \hspace{1cm}
\subfigure[Error]{\includegraphics[scale=0.3]{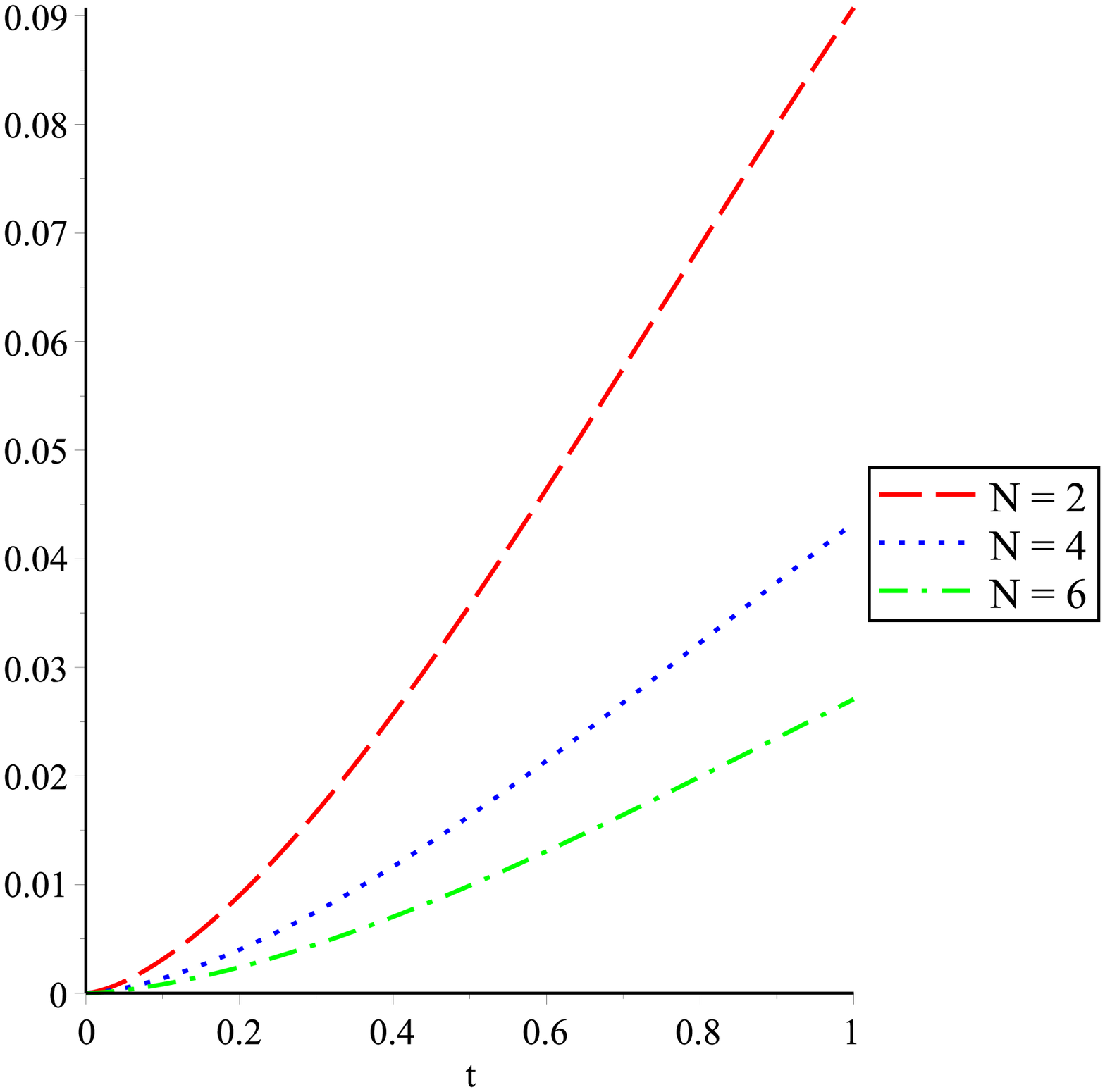}}
\end{center}
\caption{Type II left Caputo derivative of order $\beta(t)$
for the example of Section~\ref{sec:EX}---analytic
versus numerical approximations obtained from Theorem~\ref{teo3}.}
\label{IntExp7}
\end{figure}


\section{Applications}
\label{sec:app}

In this section we apply the proposed technique to some concrete
fractional differential equations of physical relevance.


\subsection{A time-fractional diffusion equation}
\label{example1}

We extend the one-dimensional time-fractional diffusion equation \cite{Lin}
to the variable order case. Consider
$u=u(x,t)$ with domain $[0,1]^2$. The partial fractional differential
equation of order $\a$ is the following:
\begin{equation}
\label{eq:diffeq}
{^C_{0}\mathbb{D}_{t}^{\a}u}(x,t)-\frac{\partial^2u}{\partial x^2}(x,t)
=f(x,t) \quad \mbox{ for } \, x\in[0,1],
\quad t\in [0,1],
\end{equation}
subject to the boundary conditions
\begin{equation}
\label{eq:diffeq:bc1}
u(x,0)=g(x), \quad \mbox{ for } \, x\in(0,1),
\end{equation}
and
\begin{equation}
\label{eq:diffeq:bc2}
u(0,t)=u(1,t)=0,\quad \mbox{ for } \, t\in[0,1].
\end{equation}
We mention that when $\a\equiv1$, one obtains the classical diffusion equation,
and when $\a \equiv 0$ one gets the classical Helmholtz elliptic equation.
Using Lemma~\ref{LemmaEx}, it is easy to check that
$$
u(x,t)=t^2\sin(2\pi x)
$$
is a solution to \eqref{eq:diffeq}--\eqref{eq:diffeq:bc2} with
$$
f(x,t)=\left(\frac{2}{\Gamma(3-\a)}t^{2-\a}
+4\pi^2t^2\right)\sin(2\pi x)
$$
and
$$
g(x)=0
$$
(compare with Example~1 in \cite{Lin}).
The numerical procedure is the following: replace ${^C_{0}\mathbb{D}_{t}^{\a}u}$
with the approximation given in Theorem~\ref{teo1},
taking $n=1$ and an arbitrary $N\geq 1$, that is,
$$
{^C_{0}\mathbb{D}_{t}^{\a}u}(x,t)\approx A t^{1-\a}
\frac{\partial u}{\partial t}(x,t) +\sum_{p=1}^N B_p t^{1-p-\a} V_p(x,t)
$$
with
\begin{equation*}
\begin{split}
A& =\DS \frac{1}{\Gamma(2-\a)}\left[
1+\sum_{l=1}^N \frac{\Gamma(\a-1+l)}{\Gamma(\a-1)l!}  \right],\\
B_p & =  \DS\frac{\Gamma(\a-1+p)}{\Gamma(1-\a)\Gamma(\a)(p-1)!},\\
V_p(x,t)& = \DS\int_{0}^{t}\t^{p-1}\frac{\partial u}{\partial t}(x,\t)d\t.
\end{split}
\end{equation*}
Then, the initial fractional problem \eqref{eq:diffeq}--\eqref{eq:diffeq:bc2}
is approximated by the following system of second-order partial differential equations:
$$
A t^{1-\a}\frac{\partial u}{\partial t}(x,t)
+\sum_{p=1}^N B_p t^{1-p-\a} V_p(x,t)
-\frac{\partial^2u}{\partial x^2}(x,t)=f(x,t)
$$
and
$$
\frac{\partial V_p}{\partial t}(x,t)=t^{p-1}
\frac{\partial u}{\partial t}(x,t),
\quad p=1,\ldots, N,
$$
for $x\in[0,1]$ and for $t\in [0,1]$,
subject to the boundary conditions
$$
u(x,0)=0, \quad \mbox{ for } \, x\in(0,1),
$$
$$
u(0,t)=u(1,t)=0,\quad \mbox{ for } \, t\in[0,1],
$$
and
$$
V_p(x,0)=0, \quad \mbox{ for } \, x\in[0,1], \quad p=1,\ldots, N.
$$

\begin{Remark}
As was mentioned in Theorem~\ref{teo1}, as $N$ increases,
the error of the approximation decreases and the given approximation
formula converges to the fractional derivative. Thus, in order to have
a good accuracy for the method, one should take higher values for $N$.
\end{Remark}

\begin{Remark}
We are not aware of similar methods to our, concerning variable fractional calculus,
in order to compare the performance of the proposed method to other numerical
approximation methods. For this reason, we decided to compare with the exact solution.
In the available literature, using a discretization process, FDEs are solved
as finite differences. Our technique is quite different: we rewrite the FDE as a system
of ordinary differential equations, and then we can apply any known technique to solve it.
Note that the reason why we stopped here with $N=6$ was to have an approximation
that is enough close to the exact solution but still visually distinguishably
(when we increase $N$ more, the approximation and the exact solution
appear to be the same in the plots). In terms of performance of the method,
it is roughly the same to put $N = 6$ or bigger.
\end{Remark}


\subsection{A fractional partial differential equation in fluid mechanics}
\label{sec:fluid:mech}

We now apply our approximation techniques to the following
one-dimensional linear inhomogeneous fractional Burgers'
equation of variable order (see \cite[Example~5.2]{Odibat}):
\begin{equation}
\label{eq:fluid:mech}
{^C_{0}\mathbb{D}_{t}^{\a}u}(x,t)+\frac{\partial u}{\partial x}(x,t)
-\frac{\partial^2u}{\partial x^2}(x,t)=\frac{2t^{2-\a}}{\Gamma(3-\a)}
+2x-2, \quad \mbox{ for } \, x\in[0,1], \quad t\in [0,1],
\end{equation}
subject to the boundary condition
\begin{equation}
\label{eq:fluid:mech:bc}
u(x,0)=x^2, \quad \mbox{ for } \, x\in(0,1).
\end{equation}
Here,
$$
F(x,t)=\frac{2t^{2-\a}}{\Gamma(3-\a)}+2x-2
$$
is the external force field. Burgers' equation is used to model gas dynamics,
traffic flow, turbulence, fluid mechanics, etc. The exact solution is
$$
u(x,t)=x^2+t^2.
$$
The fractional problem \eqref{eq:fluid:mech}--\eqref{eq:fluid:mech:bc}
can be approximated by
$$
A t^{1-\a}\frac{\partial u}{\partial t}(x,t) +\sum_{p=1}^N B_p t^{1-p-\a} V_p(x,t)
+\frac{\partial u}{\partial x}(x,t)-\frac{\partial^2u}{\partial x^2}(x,t)
=\frac{2t^{2-\a}}{\Gamma(3-\a)}+2x-2
$$
with $A$, $B_p$ and $V_p$, $p\in\{1,\ldots,N\}$, as in Section~\ref{example1}.
The approximation error can be decreased as much as desired by increasing
the value of $N$.


\small


\section*{Acknowledgments}

This work was supported by Portuguese funds through the
\emph{Center for Research and Development in Mathematics and Applications} (CIDMA)
and \emph{The Portuguese Foundation for Science and Technology} (FCT),
within project UID/MAT/04106/2013. Tavares was also supported
by FCT through the Ph.D. fellowship SFRH/BD/42557/2007;
Torres by project PTDC/EEI-AUT/1450/2012, co-financed by FEDER under
POFC-QREN with COMPETE reference FCOMP-01-0124-FEDER-028894.
The authors are very grateful to three anonymous referees,
for several comments and improvement suggestions.



\end{document}